\documentclass[reqno,10pt]{amsart}

\usepackage[left=2cm, top=3cm, right=1.5cm, bottom=3cm]{geometry}
\usepackage[utf8]{inputenc}
\usepackage{amssymb,bbm}
\usepackage{amsfonts}
\usepackage{amsmath, amscd,amsthm}
\usepackage{graphicx}
\usepackage{color}
\usepackage{enumerate}

\newtheorem{lemma}{Lemma}
\newtheorem{proposition}{Proposition}

\newtheorem{claim}{Claim}
\newtheorem{theorem}{Theorem}
\newtheorem{corollary}{Corollary}

\begin{document}
	
	\title{Scaling limit of an equilibrium surface under the Random Average Process} 
	
	
	\author{Luiz Renato Fontes}
	
	\thanks{LRF was partially supported by CNPq grant 307884/2019-8 and FAPESP grant 2017/10555-0}
		
	\author{Mariela Pent\'on Machado}
	
    \thanks{MPM was supported by FAPESP fellowship 2020/02662-4}

	\author{Leonel Zuazn\'abar}
		
	\thanks{LZ was supported by a CAPES/PNPD grant
88882.315481/2013-01 post doctoral fellowship}

	\address{
		\newline
		Luiz Renato Fontes
		\newline
		Universidade de S\~ao Paulo,
		\newline  R. do Mat\~ao, 1010 - Butant\~a, S\~ao Paulo - SP, CEP: 05508-090, Brasil.
		\newline
		e-mail: {\rm \texttt{lrfontes@usp.br}}
		\newline
		\newline
		\newline
		Mariela Pent\'on Machado
		\newline
		Universidade de S\~ao Paulo,
		\newline  R. do Mat\~ao, 1010 - Butant\~a, S\~ao Paulo - SP, CEP: 05508-090, Brasil.
		\newline
		e-mail: {\rm \texttt{marantcha@gmail.com}}
		\newline
		\newline
		Leonel Zuazn\'abar
		\newline
		Universidade de S\~ao Paulo,
		\newline  R. do Mat\~ao, 1010 - Butant\~a, S\~ao Paulo - SP, CEP: 05508-090, Brasil.
		\newline
		e-mail: {\rm \texttt{lzuaznabar@ime.usp.br}}
	}
	
	\begin{abstract}
		We consider the equilibrium surface of the Random Average Process started from an inclined plane, as seen from the height of the origin, obtained in~\cite{FerrariFontes}, where its fluctuations were shown to be of order of the square root of the distance to the origin in one dimension, and the square root of the log of that distance in two dimensions (and constant in higher dimensions). 
		Remarkably, even if not pointed out explicitly in~\cite{FerrariFontes}, the correlation structure of those fluctuations is given in terms of the Green's function of a certain random walk, and thus corresponds to those of Discrete Gaussian Free Fields.
		In the present paper we obtain the scaling limit of those fluctuations in one and two dimensions, in terms of Gaussian processes, in the sense of finite dimensional distributions.
		In one dimension, the limit is given by Brownian Motion; in two dimensions, we get a
		process with a discontinuous covariance function.
	\end{abstract}	

	\maketitle

	
	\smallskip
	

	\section{Introduction}
	
	This paper may be seen as a followup to~\cite{FerrariFontes}, even if after a long span. In the latter paper, the Random Average Process (RAP) was introduced as a dynamical random surface/field, whose heights, indexed by $d$-dimensional (discrete) space, evolve in discrete time by taking 
	averages of neighboring heights\footnote{\cite{FerrariFontes} also considers a continuous time version of the RAP.}. The average weights are random, hence the terminology. The initial condition is important for the behavior of the dynamics, and in~\cite{FerrariFontes} the case of an inclined hyperplane was considered, and, among other results on the time asymptotics of the RAP, a CLT for the height at the origin, as well as the existence 
	of  a limiting surface as seen from the height of the origin, were established.
	
	Since that initial paper, considerable attention has been devoted to that model. We mention~\cite{KG,Sch,ZS1,FMV,ZS2,BRS,CKMM1,CKMM2,DK,GRPIB}.
	
	In the present paper, we consider the above mentioned limiting/invariant surface obtained in~\cite{FerrariFontes}, and obtain the (full) scaling limit of its fluctuations in dimensions one and two (where they are unbounded; they are bounded in higher dimensions). In~\cite{FerrariFontes}, the order of magnitude of those fluctuations were shown to be the square root of the distance to the origin in one dimension, and the square root of the log of that distance in dimension two. The (limiting) shape of the fluctuations was not not addressed, and we seek to complete the picture now.
	
	A remarkable feature of the fluctuations of the limiting surface obtained in~\cite{FerrariFontes} is that its correlation structure is given by 
	the Green's function of a certain random walk (whose jump distribution depends on the weights of the random averages of the RAP). In this
	sense, there is a relation with (Discrete) Gaussian Free Fields (Gaussian fields indexed by ${\mathbb Z}^d$ with the same correlation
	 structure).\footnote{But this was not pointed out explicitly in~\cite{FerrariFontes}.}
	
	It is natural then to ask whether the limiting surface is Gaussian or not. 
	A positive answer would reduce the efforts in this paper to a straightforward computation of scaled covariances. 
	Analysis of a simpler case (in one dimension, and where the above mentioned random walk is simple) suggests that this is not the case in general, and probably never for weights of bounded range (which is the case we address in this paper).
	
	The processes we obtain as scaling limits of the fluctuations of the invariant surface under the RAP (in one and two dimensions, as aforementioned) are Gaussian. In one dimension, it is Brownian motion, and in two dimensions it has a discontinuous covariance function. Our results are in the sense of finite dimensional distributions only, in both cases. In the one dimensional case, it is conceivable that this may be strengthened to convergence in the usual space of continuous trajectories (although preliminary computations indicate quite a laborious effort on an attempt at verifying classic tighness condtions for that).  The two dimensional scaling limit cannot be continuous, so we would seem to be limited in options for going beyond
	our present results. That would of course be also the case with the two dimensional Discrete Gaussian Free Field with the same covariance, and an investigation of a possible connection with the (continuous) two dimensional Gaussian Free Field would suggest itself, even if that sounds to the authors like an uncertain project at the moment.
	
	Before closing this introductory discussion, and moving to the details of the RAP and our results, it is perhaps worth mentioning that scaling limit results were obtained in~\cite{BRS}, for the RAP in one dimension, along a characteristic direction, with a possibly random initial surface.
	We failed to find a connection to our results, even if merely in a broad or conceptual  sense, but that may be due to a lack of depth in our search.

	A final point concerns previous attempts at proving our results. One reason for the long delay since~\cite{FerrariFontes} may be traced to an unsuccessful previous approach, aiming at verifying conditions in the literature for the CLT for processes with stationary increments, such as the object of this paper in the one dimensional case. Common such conditions involve obtaining good estimates on decay of correlations,  but these seem inadequate to deal with our case, which is more amenable to second moment estimation. The only set of conditions we found previously involving moments came from~\cite{Hall&Hyde}, a reference used in~\cite{FerrariFontes} for a Martingale CLT.  It so happens that these conditions all require control over second moments of quantities involving certain conditional expectations, with some liberty over which $\sigma$-algebra to condition over. With such a choice made, one has to compute the limits of two different quantities, which have to come out equal. On and off over the years, time was spent making such (quite laborious and intricate) computations, with different choices of $\sigma$-algebra, and more recently we became convinced that this approach would not work (whatever the choice of $\sigma$-algebra), so, after some further investigation, we came up with the present, more direct approach (involving nonetheless verifying the conditions of a CLT from~\cite{Hall&Hyde}, but a Martingale CLT, as in~\cite{FerrariFontes}). Curiously, in the present approach we also need, at a point, to compute two limits which have to come out equal, and, in this case, they do.
	It is also perhaps remarkable that the present approach goes through a CLT for the process in order to obatin a CLT for the fluctuations of (one of) its invariant measure(s).
	
	
	

	\subsection{Notation}\label{Notation}
	
	Let us denote by $\{u_n(i,i+\cdot), \, n\geq 1, \, i \in \mathbb{Z}^{d}\}$ a collection of i.i.d.  random probability vectors distributed in $[0,1]^{\mathbb{Z}^{d}}$ with  finite range, and by  $\mathcal{F}_n$ the $\sigma$-algebra generated by $\{u_i,\,1\le i\leq n\}$. As in \cite{FerrariFontes}, we also asume that
	\begin{equation*}
	\mathbb{E}[u_1(0,j)] > 0,\,  \text{ for } |j| \le 1.
	\end{equation*}
	By $\{X_n(i), i \in\mathbb{Z}^d, n \ge 0 \}$, we refer to the discrete-time version of the Random Average Process (RAP) defined in (2.3) at \cite{FerrariFontes} as follows
	\begin{equation*}
	X_n(i) = \sum_{j\in\mathbb{Z}^d}u_n(i,j)X_{n - 1}(j), \text{ for } n\ge 1 \text{ and } i\in\mathbb{Z}^d.
	\end{equation*}
	Given $\lambda \in \mathbb{R}^d$,  let us denote by $\widehat{X}_{\infty}$ the weak limit of the RAP seen from the height at the origin with the initial configuration being a hyper-plane, that is
	\begin{equation}\label{limit_process}
		(\widehat{X}_{\infty}(x))_{x\in \mathbb{Z}^d} \overset{d}{=} \lim_{n\to\infty}\left(X_n(x) - X_n(0)\right)_{x\in \mathbb{Z}^d}, \text{ where } X_0(x) = x\lambda^*.
	\end{equation}
	The existence of $\widehat{X}_{\infty}$ is proved in Corollary $5.2$ in \cite{FerrariFontes}. As pointed out  in \cite{FerrariFontes},
	we may consider  the random walk in a random environment $\tilde{Y}^x_k$, with $\tilde{Y}_0^x = x$ and conditional probability transitions
	\begin{equation*}
	\mathbb{P}\left(\tilde{Y}_k^x, = j\big|\tilde{Y}_{k-1}^x = i,\mathcal{F}_n\right) = u_{n - k}(i,j), \, \text{ for } 1\le k\le n,
	\end{equation*}
	such that for every $n\geq1$
	\begin{equation}\label{conditional_representation}
	\big(X_n(x)\big)_{x\in\mathbb{Z}^{d}} \overset{d}{=} \Big(\mathbb{E}\big[X_0(\tilde{Y}_n^x)\big|\mathcal{F}_n\big]\Big)_{x\in\mathbb{Z}^{d}} .
	\end{equation}	
	Through \eqref{conditional_representation} we can get the following representation (see (2.21) and (2.22) in \cite{FerrariFontes}), that will be crucial in our way of dealing with $\widehat{X}_{\infty}$:
	\begin{equation}\label{series_representation}
	\widehat{X}_{\infty}(x)-x\lambda^*\overset{d}{=}\sum^{\infty}_{i=1}\left( W^{x}_i-W^{0}_i\right)\lambda^* , \text{ for } x\in\mathbb{Z}^d,
	\end{equation}
	where
	\begin{equation*}
		W_i^x = \mathbb{E}[\theta_i(\tilde{Y}^x_{i-1})|\mathcal{F}_i], \text{ for } i\ge 1, x \in\mathbb{Z}^d,   
	\end{equation*}
	and 
	\begin{equation*}
		\theta_n(k)=\sum_{j\in\mathbb{Z}^{d}}(j-k)u_n(k,j), \, \text{ for } n\ge 1, k\in\mathbb{Z}^d.
	\end{equation*}	
We assume that
\begin{equation*}
\sigma^2 := \mathbb{V}(\theta_1(0)\lambda^*)\in(0,\infty),
\end{equation*}	
where $\mathbb{V}$  stands for variance.

 \subsection{Main results}\label{Main results}
 We state now the main results of this work. In Theorem \ref{clt} below, we enunciate  a Central Limit Theorem for the process $\{\widehat{X}_{\infty}(x),\,x\in \mathbb{Z}^d\}$ in dimensions $d=1,2$, and in Theorem \ref{Thm_2}, its finite dimension convergence. Both theorems are obtained (directly or through corollaries) from Proposition \ref{proposition_1} and Proposition \ref{proposition_2}, also stated in this section.
 
 \begin{theorem}\label{clt} Let $\widehat{X}_{\infty}$ be the weak limit of the RAP seen from the height at the origin and with the initial configuration being a hyper-plane  as defined in \eqref{limit_process}. There exists a positive constant $c=c(d)$ such that
 	\begin{equation*}
 	\frac{\widehat{X}_{\infty}(x)-x\lambda^*}{\sqrt{\mathcal{P}_x}}\overset{d}{\underset{|x|\rightarrow \infty}{\longrightarrow}}\mathcal{N}(0,c),
 	\end{equation*}
 	where $\mathcal{P}_x=|x|$ for $d=1$, $\mathcal{P}_x=\log|x|$ for $d=2$ and $\mathcal{N}(0,c)$ is a mean zero Gaussian r.v. with variance $c$.
 \end{theorem}

Propositions \ref{proposition_1} and \ref{proposition_2}, stated next, allow us to split the infinite series in \eqref{series_representation} in two sums, such that, after normalization, one sum converges to a Gaussian law, and the second moment of the other is close to zero in a certain way to be made precise, leading to 
Theorem \ref{clt}.

 \begin{proposition}\label{proposition_1} Let us consider $\mathcal{P}_x$ as defined in Theorem \ref{clt}. The following limits exist
 	\begin{align*}
 	 h(A) &:= \lim_{|x|\to\infty}  \frac{1}{\mathcal{P}_x}\mathbb{E}\Big(\sum^{A|x|^2}_{i=1}(W^x_i-W^0_i)\lambda^*\Big)^2, \, \text{ for any } A \geq 1,\\
 	 c(d) &:= \lim_{A\to\infty} h(A) .
 	\end{align*}

 \end{proposition}

 \begin{proposition}\label{proposition_2}Let us consider $\mathcal{P}_x$ 
 	as in Theorem \ref{clt} and $h$ as in Proposition \ref{proposition_1}. Then
 	\begin{equation*}
 	\sum^{A|x|^2}_{i=1}\frac{\left(W^x_i-W^0_i\right)\lambda^*}{\sqrt{\mathcal{P}_x}} \overset{d}{\underset{{|x|\rightarrow \infty}}{\longrightarrow}}\mathcal{N}(0,h(A)) , \, \text{ for any } A \ge 1.
 	\end{equation*}
 \end{proposition}

In Section \ref{sec_corollary_1} and Section \ref{sec_corollary_2}, we obtain corollaries from Proposition \ref{proposition_1} and Proposition \ref{proposition_2}, respectively. These results lead us to Theorem \ref{Thm_2}  as Proposition 1 do it for Theorem \ref{clt}. For the sake of simplicity, we decide not to include the statement of these corollaries in this section.

\begin{theorem}\label{Thm_2}Let us consider the following rescaled process. 
		\begin{enumerate}
			\item [$(i)$] For $d = 1$ define
			\begin{equation*}
			X_{n}(t):= \frac{\widehat{X}_{\infty}(\lfloor nt\rfloor) - \lfloor nt\rfloor \lambda}{\sqrt{c\mathcal{P}_n}}, \, \text{ for } t\ge 0 \text{ and } n \ge 1.
			\end{equation*}
			Where $\mathcal{P}_n$ and $c$ are taken as in Theorem \ref{clt}. Let $\{B(t), t\ge 0\}$ be a Standard Brownian motion. Then for $0<t_1<\dots<t_k$ we have
			\begin{equation}\label{convergence_thm_2_dim_1}
			\left(X_n(t_1),\dots,X_n(t_k)\right) \overset{d}{\underset{n\rightarrow \infty}{\longrightarrow}} \left(B(t_1),\dots,B(t_k)\right).
			\end{equation}
			\item[$(ii)$] In case $d = 2$, given $z\in \mathbb{Z}^{2}\setminus(0,0)$, let us define 
			
			\begin{equation}\label{eq_scale_dim_2}
			\tilde{x}_n(z):= (\lfloor n^{|z(1)|}\rfloor, \lfloor n^{|z(2)|}\rfloor) \, \text{ for } n\ge 1,  z\in \mathbb{Z}^2,
			 \end{equation}
			 and
			\begin{equation*}
			X_n(z) := \frac{\widehat{X}_{\infty}(\tilde{x}_n(z) - \tilde{x}_n(z)\lambda^*}{\sqrt{c\mathcal{P}_n}}, \, \text{ for } z\in \mathbb{Z}^2  \text{ and } n \ge 1.
			\end{equation*}		
			Then for $z_1,\dots,z_k$ in $\mathbb{Z}^2$ we have 
			\begin{equation}\label{convergence_thm_2_dim_2}
				\left(X_n(z_1),X_n(z_2),\dots,X_n(z_k)\right) \overset{d}{\underset{n\rightarrow \infty}{\longrightarrow}}(Z_1, \dots, Z_k),
			\end{equation}
			where $(Z_1, \dots, Z_k)$ is a Gaussian vector with covariance matrix $ \left(C_{j,l}\right)_{1\le j,l\le k}$, defined as follows
			\begin{equation*}
				C_{j,l} :=\left\lbrace\begin{array}{lc} \max\{|z_j(1)|,|z_j(2)|\}, & \text{ for } j = l,\\
			\frac{1}{2}\min \{\max\{|z_l(1)|,|z_l(2)|\},\max\{|z_j(1)|,|z_j(2)|\}\}, & \text{ for } j \neq  l.
			\end{array}\right.
			\end{equation*}
		\end{enumerate}				
\end{theorem}

\bigskip
	
To conclude this introduction, we establish that the structure of the article is the following: In Section \ref{Section2}, we prove Proposition \ref{proposition_1}. In Section \ref{sec_corollary_1}, we state and prove a corollary of Propositoin \ref{proposition_1}: Corollary \ref{corollary_1}. In Section \ref{Section3}, we prove Proposition 2. In Section \ref{sec_corollary_2},  we ennunciate and prove a corollary of Proposition \ref{proposition_2}: Corollary \ref{corollary_2}. From Proposition \ref{proposition_1} and Proposition \ref{proposition_2}, in Section \ref{sec_proof_Thm_1}, we obtain Theorem \ref{clt}. From Corollary \ref{corollary_1} and Corollary \ref{corollary_2}, in Section \ref{sec_proof_of_Thm_2}, we obtain Theorem \ref{Thm_2}. Finally, in Appendix \ref{appendix_1} we present some technical calculations needed for the proof of Proposition \ref{proposition_2}. These calculations are inspired by and are very similar to ones that appear in the proof of Theorem $4.1$ in \cite{FerrariFontes}. We include them for the sake of completeness.	
	
\section{Proof of Proposition \ref{proposition_1}}\label{Section2}
	 We split  the proof of Proposition \ref{proposition_1} into two cases, $d=1$ and $d=2$. Each case will be dealt with in separate subsections, \ref{dim_1} and \ref{dim_2}, respectively. Later in Section \ref{sec_corollary_1}, we prove a corollary of Proposition \ref{proposition_1} (Corollary \ref{corollary_1}) that is used to obtain Theorem \ref{Thm_2} in the same way that Proposition \ref{proposition_1} is used to get Theorem \ref{clt} .
	 
	 	As in \cite{FerrariFontes} let  us denote by $D = \{D_n, n\ge 0\}$ and $H = \{H_n, n\ge 0\}$ two Markov chains in $\mathbb{Z}^d$ with the following transition probabilities
	 	\begin{equation*}
	 	\mathbb{P}\left(D_{n+1} = k| D_n = l\right) = \sum_{j\in\mathbb{Z}^d}\mathbb{E}\left[u_1(0,j)u_1(l,j + l)\right]\text{ and }\mathbb{P}\left(H_{n+1} = k| H_n = l\right) = \sum_{j\in\mathbb{Z}^d}\mathbb{E}\left[u_1(0,j)\right]\mathbb{E}\left[u_1(l,j + l)\right].
	 	\end{equation*} 
	 	Also let us consider the following stopping time
	
	\begin{equation}\label{tau}
		\tau=\inf\{k\geq 0:\, D_k=0\}.
	\end{equation}
	It follows from a standard argument using the Markov property (see (5.9) in \cite{FerrariFontes} ) that  for $x\in\mathbb{Z}^d$ we get
	\begin{equation}\label{eq11}
		\sum^n_{k=0}\left\{\mathbb{P}_0(D_k=0)-\mathbb{P}_x(D_k=0) \right\}=\sum^n_{k=0}\mathbb{P}_0(D_{n-k}=0)\mathbb{P}_{x}(\tau>k).
	\end{equation}
The following result is also used to prove Proposition \ref{proposition_1}. 

	\begin{lemma}\label{lemma_4} Let us consider 
		\begin{equation*}
			a(x) = \lim_{n\to\infty}\sum_{k = 0}^{n}\left\{\mathbb{P}_0\left(H_k = 0\right) - \mathbb{P}_x(H_k = 0)\right\} \text{ for } x\in\mathbb{Z}^d,
		\end{equation*} 	
		the potential kernel of the Markov chain $H$, and set ${\mathfrak A}=\mathbb{E}_0\big[a(D_1)\big]$. 
		Consider also the following quadratic form:
		\begin{equation}\label{Q}
			Q(\theta)=\mathbb{E}_0\left[\left(H_1\cdot\theta\right)^2\right] = \theta\cdot Q \cdot \theta^t, \, \theta\in\mathbb{Z}^d.
		\end{equation}
		When $d = 1$, $Q = \mathbb{E}_0\left[H_1^2\right]=:\sigma_H^2$, and when $d = 2$, $Q$ is the convariance matrix of $H_1$ with the chain starting from the origin. 
		Then, in the case $d = 1$ we have that 
		\begin{equation}\label{eq12}
			\mathbb{P}_0(D_n=0)\sim \frac{1}
			{\mathfrak A}
			\frac{1}{\sqrt{2\pi}\,\sigma_H }\frac1{\sqrt{n}} \,  \text{ as } n\to\infty,
		\end{equation}
		and when $d = 2$
		\begin{equation}\label{eq11**}
		\sum^n_{k=0}\mathbb{P}_0(D_k=0)\sim \frac{1}
    	{\mathfrak A}	
		\frac{1}{2\pi\sqrt{det(Q)}}\ln (n) \, \text{ as } n\to\infty.
		\end{equation}
	\end{lemma}
	
	\begin{proof}
		By P7.9 (pag. 75) in \cite{Spitzer} we have  that
		\begin{equation}\label{eq10}
		\lim_{n\rightarrow \infty}(2\pi n)^{d/2}\mathbb{P}_0(H_n=0)=\frac{1}{\sqrt{det(Q)}}.
		\end{equation}
		This implies that
		\begin{equation}\label{losH}
		\sum^{n}_{k=0}\mathbb{P}_0(H_k=0)\sim\left\lbrace\begin{array}{lc}\frac{2}{\sqrt{2\pi}\,\sigma_H }\sqrt{n} ,& \text{ when }d=1,\\
		\frac{1}{2\pi\sqrt{det(Q)}}\ln n, &\text{ when } d=2.
		\end{array}\right.
		\end{equation}
		Let 
			\begin{equation*}
			f(s) = \sum_{n\ge 0}\mathbb{P}\left(D_n = 0|D_0 = 0\right)s^n \text{ and } g(s) = \sum_{n\ge 0}\mathbb{P}\left(H_n = 0|H_0 = 0\right)s^n
			\end{equation*}
			be the power series of $\mathbb{P}(D_n = 0|D_0 = 0)$ and $\mathbb{P}(H_n = 0|H_0 = 0)$, respectively. By Theorem $5$ in \cite{Feller} (page 447) and \eqref{losH},  we obtain
		\begin{equation}
		g(s)\sim\left\lbrace\begin{array}{lc}\frac{1}{\sqrt{2}\,\sigma_H} \frac{1}{\sqrt{1 -s}}, & \text{ when }d=1,\\
		\frac{1}{2\pi\sqrt{det(Q)}}\ln \left(\frac{1}{1-s}\right), & \text{ when }d=2
		\end{array}\right.
		\text{ as } s\to 1^-.
		\end{equation}
		In Lemma 3.2 of \cite{FerrariFontes} is proved that 
		\begin{equation}\label{eq_206}
		\lim_{s\rightarrow 1}\frac{f(s)}{g(s)}=\frac{1}
		{\mathfrak A}. \footnotemark
		\end{equation}
	\footnotetext{\eqref{eq_206} corresponds to Equation 3.14 in~\cite{FerrariFontes}, except that there are (minor)  mistakes in the argument in~\cite{FerrariFontes} leading to
	the latter equation, causing the appearance of the extraneous factor of $1-\gamma'$ in (3.14) of ~\cite{FerrariFontes}, which should not be there, and also causing the expression $(1-\gamma)\sum_xa(x)p_x$, which we presently write as ${\mathfrak A}={\mathbb{E}_0(a(D_1))}$, to appear in the numerator (so to say), rather than in the denominator.}
		Therefore, 
		\begin{equation}
		f(s)\sim\left\lbrace
		\begin{array}{lc}\frac{1}
		{\mathfrak A}
			\frac{1}{\sqrt{2}\,\sigma_H}\frac{1}{\sqrt{1 - s}}, & \text{ when }d=1,\\
		\frac{1}
		{\mathfrak A}
		\frac{1}{2\pi\sqrt{det(Q)}}\ln \left(\frac{1}{1-s}\right), & \text{ when }d=2
		\end{array}\right.
		 \text{ as } s\to 1^-. 
		\end{equation}
		Again by Theorem 5 in \cite{Feller} we get \eqref{eq11**}. In Lemma 3.1 of \cite{FerrariFontes} is proved that $\mathbb{P}_0(D_n=0)$ is non-increasing in $n$. Therefore, we can use the second part of  Theorem $5$ in \cite{Feller} for the case $d=1$, and we obtain \eqref{eq12}.
	\end{proof}	
	One last comment before the proof Proposition \ref{proposition_1} is that  by equation is (5.8) in \cite{FerrariFontes} we have that
	\begin{equation}\label{eq_104}
		\mathbb{E}\Big(\sum_{i = 1}^{A|x|^2}\left(W_i^x - W_i^0\right)\lambda^*\Big)^2 = 2\sigma^2\sum_{i = 1}^{A|x|^2}\left\{\mathbb{P}(D_i=0|D_0=0)-\mathbb{P}(D_i=0|D_0=x)\right\}.
	\end{equation}
	Hence, to get Proposition \ref{proposition_1}, it is enough to show that
	\begin{enumerate}
		\item[(i)] when $d = 1$,
		\begin{equation}\label{eq8*}
		\frac{2\sigma^2}{|x|}\sum^{Ax^2}_{k=1}\left\{\mathbb{P}(D_k=0|D_0=0)-\mathbb{P}(D_k=0|D_0=x)\right\}\underset{|n|\rightarrow\infty}{\longrightarrow} h(A)\underset{A\rightarrow\infty}{\longrightarrow} c,
	\end{equation}
		\item[(ii)] when $d = 2$,
		\begin{equation}\label{eq8**}
		\frac{2\sigma^2}{\log|x|}\sum^{A|x|^2}_{k=1}\left\{\mathbb{P}(D_k=0|D_0=0)-\mathbb{P}(D_k=0|D_0=x)\right\}\underset{|x|\rightarrow\infty}{\longrightarrow}c, \text{ for all } A\ge 1.
		\end{equation}
	\end{enumerate}
	
	\subsection{Proof of Proposition \ref{proposition_1} when $\mathbf{d=1}$}\label{dim_1}
	
	Since the chains $D$ and $H$  are symmetric (see Lemma 2.5 in \cite{FerrariFontes}), we may consider~\eqref{eq8*} with $x=n>0$.
   By \eqref{eq11} and Lemma \ref{lemma_4}, the left hand side of \eqref{eq8*} becomes 
	\begin{equation}\label{eq8}
	\frac{c'}{n}\sum^{An^2}_{k=1}\frac{1}{\sqrt{An^2-k}}\mathbb{P}_n(\tau>k),\,\text{ where } 
	c'=\frac{2\sigma^2}{{\mathfrak A}\sqrt{2\pi}\sigma_H }. 
	\end{equation}
	The left hand side  of \eqref{eq8} then equals
	\begin{equation}\label{eq15}
	c'\sum^{An^2}_{k=1}\frac{1}{\sqrt{A-\frac{k}{n^2}}}\mathbb{P}_0\left(\frac{\tau_n}{n^2}>\frac{k}{n^2}\right)\frac{1}{n^2},
	\end{equation} 
	where $\tau_n$ is the hitting time of $n$ by $H$.
	It follows from well know facts 
	\begin{equation}\label{eq22}
	\mathbb{P}_0\left(\frac{\tau_n}{n^2}>t\right)\underset{n\rightarrow\infty}{\longrightarrow}\mathbb{P}(T_{\frac{1}{\sigma_H}}>t),
	\end{equation}
	uniformly on compact intervals, where the random variable $T_a$ is the passage time of $a$ by a standard one-dimensional Brownian motion,
	$a\in\mathbb R$.

	By  \eqref{eq22}, the expresion in \eqref{eq15}  is asymptotically equivalent to 
	\begin{equation}
	c'\sum^{An^2}_{k=1}\frac{1}{\sqrt{A-\frac{k}{n^2}}}\mathbb{P}\left(T_{1/\sigma_H}>\frac{k}{n^2}\right)\frac{1}{n^2}.
	\end{equation}
	We now observe that
	\begin{equation}\label{eq_100}
		\lim_{n\to\infty}c'\sum^{An^2}_{k=1}\frac{1}{\sqrt{A-\frac{k}{n^2}}}\mathbb{P}\left(T_{1/\sigma_H}>\frac{k}{n^2}\right)\frac{1}{n^2} = c'\int^{A}_{0}\frac{1}{\sqrt{A-y}}\mathbb{P}(T_{1/\sigma_H}>y)dy=c'\int^{1}_{0}\frac{\sqrt{A}\mathbb{P}(T_{1/\sigma_H}>Ay)}{\sqrt{1-y}}dy,
	\end{equation}
which proves the first part of Proposition \ref{proposition_1}, namely
\begin{equation}\label{eq_123}
 \lim_{|x|\to\infty} \frac{1}{\mathcal{P}_x}\mathbb{E}\left(\sum^{A|x|^2}_{i=1}(W^x_i-W^0_i)\lambda^*\right)^2= c'\int^{1}_{0}\frac{\sqrt{A}\mathbb{P}(T_{1/\sigma_H}>Ay)}{\sqrt{1-y}}dy=:h(A).
\end{equation}
 Hence, to conclude the proof, we verify the following claim. 
	\begin{claim}\label{claim_2} 
		\begin{equation}
			\lim_{A\to\infty}\int^{1}_{0}\frac{\sqrt{A}\,\mathbb{P}(T_{1/\sigma_H}>Ax)}{\sqrt{1-x}}dx = \frac{\sqrt{2\pi}}{\sigma_H}.
		\end{equation}
	\end{claim}
	\begin{proof}[Proof of Claim \ref{claim_2}]
		From a well known formula\footnote{See, e.g., Remark $2.8.3$ in page 92 of \cite{Karatzas}}, we have
		\begin{equation}\label{hit}
		\mathbb{P}(T_{1/\sigma_H} \leq  t)=\frac{2}{\sqrt{2 \pi} }\int^{\infty}_{\frac{1}{\sigma_H\sqrt{t}}}e^{-\frac{x^2}{2}}dx.
		\end{equation}
		Applying L'H\^opital's rule, we find that
		$$
		\lim_{t\rightarrow \infty}\frac{1-\frac{2}{\sqrt{2 \pi}}\int^{\infty}_{\frac{1}{\sigma_H\sqrt{t}}}e^{-\frac{x^2}{2}}dx}{1/\sqrt{t}}=\lim_{t\rightarrow \infty}\frac{2}{\sqrt{2\pi}\sigma_H} e^{-\frac{1}{2\sigma^2_H t}}=\frac{2}{\sqrt{2\pi }\sigma_H};
		$$
		it then follows that, for every $x\in (0,1)$
		\begin{equation}\label{eq_121}
		\lim_{A\rightarrow \infty}\sqrt{A}\mathbb{P}(T_1\geq A x)=\frac{2}{\sqrt{2\pi}\sigma_H\sqrt{x}}.
		\end{equation}
	    It also follows from~\eqref{hit} that
		\begin{equation}\label{eq_120}
		\mathbb{P}(T_{1/\sigma_H}>t)\leq \frac{1}{\sigma_H\sqrt{t}},
		\end{equation}
		for all $t> 0$. 
		We prove this by computing the derivative of 
		\begin{equation}\label{func_to_derivar}
		 \frac{1}{\sigma_H\sqrt{t}}-1+\frac{2}{\sqrt{2\pi}}\int^{\infty}_{\frac{1}{\sigma_H\sqrt{t}}}e^{-\frac{x^2}{2}}dx, 
		\end{equation}
 and checking that it is negative for all $t>0$, which implies that the expression in~\eqref{func_to_derivar} is non-increasing in $(0,\infty)$. 
	 Since $\int^{\infty}_0e^{-\frac{x^2}{2}}dx=\sqrt{2\pi}/2$, \eqref{func_to_derivar} vanishes as $t\to\infty$. Hence, the equation \eqref{func_to_derivar} is greater or equal to zero for all $t> 0$, and \eqref{eq_120} follows. 
	Now \eqref{eq_121} and \eqref{eq_120} allow us to  use the Dominate Convergence Theorem to get
		\begin{equation*}
		\int^{1}_{0}\frac{\sqrt{A}\mathbb{P}(T_{1/\sigma_H}>Ax)}{\sqrt{1-x}}dx\underset{A\rightarrow \infty}{\longrightarrow} \frac{2}{\sqrt{2\pi}\sigma_H} \int^{1}_{0}\frac{1}{\sqrt{x(1-x)}}dx=\frac{2}{\sqrt{2\pi}\sigma_H}\pi=\frac{\sqrt{2\pi}}{\sigma_H}.
		\end{equation*}

	\end{proof}
	Therefore, by \eqref{eq_123} and Claim \ref{claim_2}  we obtain
	\begin{equation*}
		\lim_{A\to\infty}\lim_{|x|\to\infty} \frac{1}{\mathcal{P}_x}\mathbb{E}\left(\sum^{A|x|^2}_{i=1}[W^x_i-W^0_i]\lambda^*\right)^2 = c'\frac{ \sqrt{2\pi}}{\sigma_H}=\frac{2\sigma^2}{{\mathfrak A}\,\sigma^2_H} =: c(1).
	\end{equation*}
	and the proof  of Proposition \ref{proposition_1} is concluded for $d=1$.
	
	\subsection{Proof of Proposition \ref{proposition_1} when $\mathbf{d=2}$}\label{dim_2}
	
	Rewriting the sum in the left hand of \eqref{eq8**} we obtain
	\begin{align}\label{eq10**}
	\sum^{A|x|^2}_{k=1}\left\{\mathbb{P}(D_k=0|D_0=0)-\mathbb{P}(D_k=0|D_0=x)\right\}=\mathbb{E}_0\left(\sum^{A|x|^2}_{k=1}\mathbbm{1}_{\{D_k=0\}}\right)-\mathbb{E}_x\left(\sum^{A|x|^2}_{k=1}\mathbbm{1}_{\{D_k=0\}}\right).
	\end{align}
	Let us work first with the second expected value in the right hand side of of \eqref{eq10**}. Consider $\tau$ as defined in \eqref{tau}. Then by the Markov property we get that
	\begin{equation}\label{eq9**}
	\mathbb{E}_x\left(\sum^{A|x|^2}_{k=1}\mathbbm{1}_{\{D_k=0\}}\right)=\mathbb{E}_x\left(\sum^{A|x|^2}_{k=\tau}\mathbbm{1}_{\{D_k=0\}};\tau<A|x|^2\right)= \sum^ {A|x|^2}_{j=1}\mathbb{E}_0\left(\sum^{A|x|^2-j}_{k=0}\mathbbm{1}_{\{D_k=0\}}\right)\mathbb{P}_x(\tau=j).
	\end{equation}
	Substituting \eqref{eq9**} into \eqref{eq10**},  we find that
	\begin{align}\label{eq12**}
	&\sum^{A|x|^2}_{k=1}\left\{\mathbb{P}(D_k=0|D_0=0)-\mathbb{P}(D_k=0|D_0=x)\right\} \nonumber\\
	&=\mathbb{E}_0\left(\sum^{A|x|^2}_{k=1}\mathbbm{1}_{\{D_k=0\}}\right)-\sum^ {A|x|^2}_{j=1}\mathbb{E}_0\left(\sum^{A|x|^2-j}_{k=0}\mathbbm{1}_{\{D_k=0\}}\right)\mathbb{P}_x(\tau=j)\nonumber\\
	&=\mathbb{P}_x\left(\tau>A|x|^ 2\right)\mathbb{E}_0\left(\sum^{A|x|^2}_{k=1}\mathbbm{1}_{\{D_k=0\}}\right)\nonumber +\sum^{A|x|^2}_{j=1}\left\lbrace\mathbb{E}_0\left(\sum^{A|x|^2}_{k=1}\mathbbm{1}_{\{D_k=0\}}\right)-\mathbb{E}_0\left(\sum^{A|x|^2-j}_{k=0}\mathbbm{1}_{\{D_k=0\}}\right)\right\rbrace\mathbb{P}_x(\tau=j)\nonumber\\
	&=\mathbb{P}_x\left(\tau>A|x|^ 2\right)\mathbb{E}_0\left(\sum^{A|x|^2}_{k=1}\mathbbm{1}_{\{D_k=0\}}\right)+\sum^{A|x|^2}_{j=1}\mathbb{E}_0\left(\sum^{A|x|^2}_{k=A|x|^2-j+1}\mathbbm{1}_{\{D_k=0\}}\right)\mathbb{P}_x(\tau=j).
	\end{align}
	Then, by  equation \eqref{eq_104}  and \eqref{eq12**}, we have that
	\begin{align}\label{eq_105}
		&\frac{1}{\log|x|}\,\mathbb{E}\left(\sum_{i = 1}^{A|x|^2}\left[W_i^x - W_i^0\right]\lambda^*\right)^2 \nonumber\\
		&=\frac{2\sigma^2}{\log|x|}\,\mathbb{P}_x(\tau>A|x|^ 2)\,\mathbb{E}_0\left(\sum^{A|x|^2}_{k=1}\mathbbm{1}_{\{D_k=0\}}\right)
		+\frac{2\sigma^2}{\log|x|}\sum^{A|x|^2}_{j=1}\mathbb{E}_0\left(\sum^{A|x|^2}_{k=A|x|^2-j+1}\mathbbm{1}_{\{D_k=0\}}\right)\mathbb{P}_x(\tau=j) \nonumber\\
		&\ge \frac{2\sigma^2}{\log|x|}\,\mathbb{P}_x(\tau>A|x|^2)\sum^{A|x|^2}_{k=1}\mathbb{P}_0(D_k=0) .
	\end{align}
		By  (2.25) in \cite{FerrariFontes}, we have that
		\begin{equation}\label{mu}
		\mathbb{E}\left(W_i^x|\mathcal{F}_{i-1}\right) = \mathbb{E}[\theta_1(0)\lambda^*]=:\mu, \, \text{for } i\ge 1 \text{ and } x\in\mathbb{Z}^d.
		\end{equation}
		From this and time independence, it follows that 
		\begin{equation}\label{eq_107}
		\mathbb{E}\left[\left(W_i^y\lambda^* - W_i^0\lambda^*\right)\left(W_j^z\lambda^* - W_j^0\lambda^*\right)\right] = 0\, \text{ for } i \ne j \text{ and } y,z\in\mathbb{Z}.
		\end{equation} 
		By \eqref{series_representation} and \eqref{eq_107}  we have
	\begin{align}\label{eq_131}
	\frac{\mathbb{E}\left(\widehat{X}_{\infty}(x) - x\lambda^*\right)^2}{\mathcal{P}_x} &=\frac{1}{\mathcal{P}_x}\,\mathbb{E}\Big(\sum^{A|x|^2}_{i=1}(W^x_i-W^0_i)\lambda^*\Big)^2+\frac{1}{\mathcal{P}_x}\,\mathbb{E}\Big(\sum^{\infty}_{i=A|x|^2+1}(W^x_i-W^0_i)\lambda^*\Big)^2 \nonumber\\
	& \ge \frac{1}{\log|x|}\,\mathbb{E}\Big(\sum^{A|x|^2}_{i=1}(W^x_i-W^0_i)\lambda^*\Big)^2.
	\end{align}
	Then, by \eqref{eq_105} and \eqref{eq_131}, we  obtain
	\begin{equation}\label{eq_124}
		\frac{2\sigma^2}{\log|x|}\,\mathbb{P}_x(\tau>A|x|^ 2)\sum^{A|x|^2}_{k=1}\mathbb{P}_0(D_k=0) 
		\leq \frac{1}{\log|x|}\,\mathbb{E}\left(\sum^{A|x|^2}_{i=1}(W^x_i-W^0_i)\lambda^*\right)^2\leq \frac{\mathbb{E}\left(\widehat{X}_{\infty}(x) - x\lambda^*\right)^2}{\log |x|}.	
	\end{equation}
	By $(5.14)$ in \cite{FerrariFontes} and Theorem 1 in \cite{fukai}, when $d = 2$ we have
		\begin{align}\label{eq_140}
		\lim_{|x|\to\infty}\frac{\mathbb{E}\left(\widehat{X}_{\infty}(x) - x\lambda^*\right)^2}{\log|x|} = 
		\frac{2\sigma^2}{
		{\mathfrak A}\pi\sqrt{det(Q)}},
		\end{align}
		with $Q$ as defined in the paragraph of~\eqref{Q}. 
		
		We now want to prove that the limit as $|x|\to\infty$ of the left hand side in \eqref{eq_124} is also equal to the right hand side of \eqref{eq_140}. In order to do that, we use Lemma \ref{lemma_4}, that implies that
	\begin{equation*}
	\sum^{A|x|^2}_{k=1}\mathbb{P}_0(D_k=0) \sim \frac{1}{{\mathfrak A} 2\pi\sqrt{det(Q)}}\log(A|x|^2) \text{ as } |x|\to\infty,
	\end{equation*}
	and by  the corollary of Theorem $1$ in \cite{rrowe},
	$$
	\mathbb{P}_x(\tau>A|x|^2)=\frac{[1+o(1)]2\log(|x|)}{\log(A|x|^2)} \text{ for all } x \ne 0 \text{ and  }A\geq 1.
	$$
Therefore,
	\begin{equation}\label{eq_132}
		\lim_{|x|\to\infty}\frac{2\sigma^2}{\log|x|}\,\mathbb{P}_x(\tau>A|x|^ 2) 
		\sum^{A|x|^2}_{k=1}\mathbb{P}_0(D_k=0) 
		= \frac{2\sigma^2}{{\mathfrak A}\pi\sqrt{det(Q)}}.
	\end{equation} 
	Then by \eqref{eq_124}, \eqref{eq_140} and \eqref{eq_132}  we conclude that  
	\begin{equation}\label{eq_76}
	\underset{|x|\rightarrow\infty}{\lim}\frac{1}{\log|x|}\mathbb{E}\left(\sum^{A|x|^2}_{i=1}(W^x_i-W^0_i)\lambda^*\right)^2
	=\frac{2\sigma^2}{{\mathfrak A}\pi\sqrt{det(Q)}}:=c(2), \, \text{ for all }A\geq 1,
	\end{equation}
	thus completing the proof of Proposition \ref{proposition_1} for $d=2$.

	\section{Corollary of Proposition \ref{proposition_1}}\label{sec_corollary_1}
	
	By the Cram\'er-Wold theorem, convergence in distribution of a sequence of random vectors is equivalent to that of  arbitrary linear combinations of its coordinates. So, in order to obtain Theorem \ref{Thm_2}, it suffices that we state and prove in this section the following result. 
	
		\begin{corollary}\label{corollary_1} Given $k\ge 1$, let $\bar{\alpha}=(\alpha_1,\dots,\alpha_k)\in\mathbb{R}^k$. Let us also consider  $c$  and $\mathcal{P}_n$ as in Propostion \ref{proposition_1}. 
		\begin{enumerate}
			\item[(i)] In case $d = 1$, for $\bar{t} = (t_1,\dots,t_k)\in\mathbb{R}^k$ with $0< t_1,\dots< t_k$,  the following limit exits
			\begin{equation*}
			g(A,\bar{t},\bar{\alpha}):=\lim_{n\to\infty}\frac{1}{c\mathcal{P}_n}\mathbb{E}\left[\Big(\sum_{j=1}^k\alpha_j\sum_{i=1}^{An^2}\left[W_i^{\lfloor nt_j\rfloor} - W_i^0\right]\lambda\Big)^2\right], \text{ for } A\ge 1,
			\end{equation*}
			and 
			\begin{equation*}
			\lim_{A\to\infty} g(A,\bar{t},\bar{\alpha}) = \sum_{j=1}^k\alpha_j^2t_j + 2\sum_{1\le j < l \le k}\alpha_j\alpha_lt_{j}.
			\end{equation*}
			
			\item[(ii)] In case $d = 2$,  for $\bar{z}=(z_1,\dots, z_k)\in(\mathbb{Z}^{2})^k$ and $\tilde{x}_n$ as defined in \eqref{eq_scale_dim_2} we have
			\begin{align*}
			g(\bar{z}, \bar{\alpha}) &:= \lim_{n\to\infty}\frac{1}{c\mathcal{P}_n}\mathbb{E}\left[\Big(\sum_{j=1}^k\alpha_j\sum_{i=1}^{ M_{k,n}}\left[W_i^{\tilde{x}_n(z_j)} - W_i^0\right]\lambda^*\Big)^2\right]\\& \hspace{0,1 cm}= \Big(\sum_{j=1}^k\alpha_j^2\max\{|z_j(1)|,|z_j(2)|\}+\sum_{1\le j < l\le k}\alpha_j\alpha_{l}\min \{\max\{|z_l(1)|,|z_l(2)|\},\max\{|z_j(1)|,|z_j(2)|\}\}\Big),
			\end{align*}
			where $M_{k,n}=\underset{ 1\leq j\leq k}{\max}|\tilde{x}_n(z_j)|^2$.
		\end{enumerate}	
	\end{corollary}
	The proof of Corollary \ref{corollary_1} contains two parts, one for each dimension, but before starting, let us point out something useful for both dimensions. For $x,y\in\mathbb{Z}$ and $\mu$ as in \eqref{mu} we have
	\begin{align}\label{eq_142}
	&\left(W_i^x\lambda^* - W_i^0\lambda^*\right)\left(W_i^y \lambda^*- W_i^0\lambda^*\right)\nonumber\\ &=\left(W_i^x \lambda^*- \mu\right)\left(W_i^y\lambda^* -\mu\right) - \left(W_i^x\lambda^* - \mu\right)\left(W_i^0\lambda^* -\mu\right)
	-\left(W_i^0\lambda^* - \mu\right)\left(W_i^y\lambda^* -\mu\right) + \left(W_i^0 \lambda^*- \mu\right)^2.
	\end{align}
	By the translation invariance of the model, and reasoning as in \cite{FerrariFontes} to get  $(5.7)$,  we obtain
	\begin{equation}\label{eq_143}
	\mathbb{E}\left[\left(W_i^x\lambda -\mu\right)\left(W_i^y\lambda -\mu\right)\right] = \mathbb{E}\left[\left(W_i^{x - y}\lambda -\mu\right)\left(W_i^0 -\mu\right)\right] =\sigma^2 \mathbb{P}\left(D_{i-1} = 0|D_0 = x -y\right).
	\end{equation}
	Using now \eqref{eq_142} and \eqref{eq_143}  we get
	\begin{align}\label{eq_160}
	\mathbb{E}\left[\left(W_i^x\lambda - W_i^0\lambda\right)\left(W_i^y \lambda- W_i^0\lambda\right)\right] &= \sigma^2\left\{ \mathbb{P}\left(D_{i-1} = 0|D_0 = x -y\right) - \mathbb{P}\left(D_{i-1} = 0|D_0 = x\right)\right\}\nonumber\\
	& - \sigma^2\left\{\mathbb{P}\left(D_{i-1} = 0|D_0 = y\right) + \mathbb{P}\left(D_{i-1} = 0|D_0 =  0 \right)\right\}.
	\end{align}
	
	\subsection{Proof of Corollary \ref{corollary_1} in $\mathbf{d=1}$.} 
	Notice that
	\begin{align}\label{eq_204}
	\mathbb{E}\left[\left(\sum_{j=1}^k\alpha_j\sum_{i=1}^{An^2}\left[W_i^{\lfloor nt_j\rfloor} - W_i^0\right]\lambda\right)^2\right] &= \sum_{j=1}^k\alpha_j^2\mathbb{E}\left[\left(\sum_{i=1}^{An^2}\left[W_i^{\lfloor nt_j\rfloor} - W_i^0\right]\lambda\right)^2\right] \nonumber\\
	&+2\sum_{1 \le j < l \le k}\alpha_j\alpha_l\sum_{i=1}^{An^2}\sum_{m=1}^{An^2}\mathbb{E}\left[\left(\left[W_i^{\lfloor nt_j\rfloor} - W_i^0\right]\lambda\right)\left(\left[W_m^{\lfloor nt_l\rfloor} - W_m^0\right]\lambda \right]\right). 
	\end{align}
	By \eqref{eq_107} we have that
	\begin{equation}\label{eq_202}
	\mathbb{E}\left[\left(\sum_{i=1}^{An^2}\left[W_i^{\lfloor nt_j\rfloor} - W_i^0\right]\lambda\right)^2\right] = \sum_{i=1}^{An^2}\mathbb{E}\left[\left(\left[W_i^{\lfloor nt_j\rfloor} - W_i^0\right]\lambda\right)^2\right]
	\end{equation}
	and 
	\begin{equation}\label{eq_203}
	\sum_{i=1}^{An^2}\sum_{m=1}^{An^2}\mathbb{E}\left[\left(\left[W_i^{\lfloor nt_j\rfloor} - W_i^0\right]\lambda\right)\left(\left[W_m^{\lfloor nt_l\rfloor} - W_m^0\right]\lambda \right]\right)  = \sum_{i=1}^{An^2}\mathbb{E}\left[\left(\left[W_i^{\lfloor nt_j\rfloor} - W_i^0\right]\lambda\right)\left(\left[W_i^{\lfloor nt_l\rfloor} - W_i^0\right]\lambda \right]\right).
	\end{equation}
	By Proposition \ref{proposition_1} and \eqref{eq_202}, we have that
	\begin{align}\label{eq_205}
		\lim_{n\to\infty}  \frac{1}{c\mathcal{P}_n} \sum_{j=1}^k \alpha_j^2\mathbb{E}\left[\left(\sum_{i=1}^{An^2}\left[W_i^{\lfloor nt_j\rfloor} - W_i^0\right]\lambda\right)^2\right] &= \lim_{n\to\infty}  \frac{1}{c\mathcal{P}_n} \sum_{i=1}^{An^2}\mathbb{E}\left[\left(\left[W_i^{\lfloor nt_j\rfloor} - W_i^0\right]\lambda\right)^2\right] \nonumber\\
		&= \frac{1}{c} \sum_{j = 1}^k\alpha_j^2t_jh(A/t_j^2) \nonumber\\
		&\to \sum_{j = 1}^k\alpha_j^2t_j, \text{ as } A\to\infty.
	\end{align}
	Hence, by \eqref{eq_204},\eqref{eq_203} and \eqref{eq_205}, to finish the proof of the corollary for dimension 1, it is enough to compute 
	\begin{equation}\label{eq_191}
	\lim_{n\to\infty}\frac{1}{\mathcal{P}_n}\sum_{i=1}^{An^2}\mathbb{E}\left[\left(\left[W_i^{\lfloor nt_j\rfloor} - W_i^0\right]\lambda\right)\left(\left[W_i^{\lfloor nt_l\rfloor} - W_i^0\right]\lambda \right]\right) \, \text{ for }   1 \le j < l\le k.
	\end{equation}
	and its limit as $A\to\infty$.
	By $\eqref{eq_160}$ we obtain
	\begin{align}\label{eq_161}
	&\sum_{i=1}^{An^2}\mathbb{E}\left[\left(W_i^{\lfloor nt_j\rfloor}\lambda - W_i^0\lambda\right)\left(W_i^{\lfloor nt_l\rfloor}\lambda - W_i^0\lambda\right)\right]\nonumber\\
	&=\sigma^2\sum_{i=1}^{An^2}\left\{\mathbb{P}\left(D_{i-1} = 0|D_0 = \lfloor nt_l\rfloor - \lfloor nt_j \rfloor\right) - \mathbb{P}\left(D_{i-1} = 0|D_0 = \lfloor nt_j \rfloor\right)\right\}\nonumber \\
	&+\sigma^2\sum_{i=1}^{An^2}\left\{- \mathbb{P}\left(D_{i-1} = 0|D_0 =\lfloor nt_l\rfloor \right) + \mathbb{P}\left(D_{i-1} = 0|D_0 = 0\right)\right\}\nonumber\\
	&=-\sigma^2\sum_{i=1}^{An^2}\left\{\mathbb{P}\left(D_{i-1} = 0|D_0 =0\right) -\mathbb{P}\left(D_{i-1} = 0|D_0 = \lfloor nt_l\rfloor - \lfloor nt_j \rfloor\right)\right\}\nonumber\\
	&+\sigma^2\sum_{i=1}^{An^2}\left\{\mathbb{P}\left(D_{i-1} = 0|D_0 = 0\right)- \mathbb{P}\left(D_{i-1} = 0|D_0 = \lfloor nt_j \rfloor\right)\right\} \nonumber\\
	&+\sigma^2\sum_{i=1}^{An^2}\left\{\mathbb{P}\left(D_{i-1} = 0|D_0 =0\right)  - \mathbb{P}\left(D_{i-1} = 0|D_0 =\lfloor nt_l\rfloor \rfloor\right)\right\},
	\end{align}
	and, from \eqref{eq_104}, \eqref{eq8*}  and \eqref{eq_161}, we find that
	\begin{align*}
		&\lim_{n\to\infty}\frac{1}{c\mathcal{P}_N}\sum_{i=1}^{An^2}\mathbb{E}\left[\left(W_i^{\lfloor nt_j\rfloor}\lambda - W_i^0\lambda\right)\left(W_i^{\lfloor nt_l\rfloor}\lambda - W_i^0\lambda\right)\right] \\
		&= -(t_l - t_j)\frac{h(A/(t_l-t_j)^2)}{2c} + t_j\frac{h(A/(t_l-t_j)^2)}{2c}  + t_l\frac{h(A/(t_l - t_j)^2)}{2c}\\
		&\to t_j \text{ as } A\to\infty,
	\end{align*}
	thus concluding the proof of Corollary \ref{corollary_1} in $d=1$.
		
	\subsection{Proof of Corollary \ref{corollary_1} in $\mathbf{d=2}$}  By \eqref{eq_204}, \eqref{eq_202} and \eqref{eq_203}, we get that
	\begin{align}\label{eq_166}
	&\frac{1}{c\mathcal{P}_n}\mathbb{E}\left[\left(\sum_{j=1}^k\alpha_j\sum_{i=1}^{ M_{k,n}}\left[W_i^{\tilde{x}_n(z_j)} - W_i^0\right]\lambda^*\right)^2\right]\nonumber\\
	&= \frac{1}{c\mathcal{P}_n}\sum_{j=1}^k\alpha^2_j\sum_{i=1}^{M_{k,n}}\mathbb{E}\left([W^{\tilde{x}_n(z_j)}_i-W^0_i]\lambda^*\right)^2\nonumber\\
	&+ \frac{2}{c\mathcal{P}_n}\sum_{1\leq j<l\leq k}\alpha_j\alpha_l\sum_{i=1}^{M_{k,n}}\mathbb{E}\left(W^{\tilde{x}_n(z_j)}_i\lambda^*-W^0_i\lambda^*\right)\left(W^{\tilde{x}_n(z_l)}_i\lambda^*-W^0_i\lambda^*\right).
	\end{align}
	Let us start with the first term in the right hand side of \eqref{eq_166}. For each fixed $j$, we have that
	\begin{align}\label{eq_167}
	\sum_{i=1}^{|\tilde{x}_n(z_j)|^2}\mathbb{E}\left(W^{\tilde{x}_n(z_j)}_i\lambda^*-W^0_i\lambda^*\right)^2\leq \sum_{i=1}^{M_{k,n}}\mathbb{E}\left(W^{\tilde{x}_n(z_j)}_i\lambda^*-W^0_i\lambda^*\right)^2\leq \sum_{i=1}^{\infty}\mathbb{E}\left(W^{\tilde{x}_n(z_j)}_i\lambda^*-W^0_i\lambda^*\right)^2.
	\end{align}
	Using equations \eqref{eq_140}, \eqref{eq_76}  and \eqref{eq_167}, we have that 
	\begin{equation*}
	 \lim_{n\to\infty}\frac{1}{c\log(|\tilde{x}_n(z_j)|)}\sum_{i=1}^{M_{k,n}}\mathbb{E}\left(W^{\tilde{x}_n(z_j)}_i\lambda^*-W^0_i\lambda^*\right)^2 = 1.
	\end{equation*}
	Hence, for every fixed $j$, we have that 
	\begin{align}\label{eq_168}
	\lim_{n\to\infty}\frac{1}{c\mathcal{P}_n}\sum_{i=1}^{M_{k,n}}\mathbb{E}\left(W^{\tilde{x}_n(z_j)}_i\lambda^*-W^0_i\lambda^*\right)^2 &= \lim_{n\to\infty} \frac{\log(|\tilde{x}_n(z_j)|)}{\mathcal{P}_n} \nonumber\\
	&=\lim_{n\to\infty}\frac{\log(n^{\max\{|z_j(1)|,|z_j(2)|\}}(1+o(1)))}{\mathcal{P}_n} \nonumber\\ &=\max\{|z_j(1)|,|z_j(2)|\}.
	\end{align}
	Therefore
	\begin{equation*}
		\lim_{n\to\infty}\frac{1}{c\mathcal{P}_n}\sum_{j=1}^k\alpha^2_j\sum_{i=1}^{M_{k,n}}\mathbb{E}\left([W^{\tilde{x}_n(z_j)}_i-W^0_i]\lambda^*\right)^2 = \sum^k_{j=1}\alpha^2_j\max\{|z_j(1)|,|z_j(2)|\}.
	\end{equation*} 
	To deal with the second term on the right member of  \eqref{eq_166} observe that, for $1\le j < l \le k$, by \eqref{eq_160}  we have 
	\begin{align}\label{eq_169}
	&\sum_{i=1}^{M_{k,n}}\mathbb{E}\left[(W^{\tilde{x}_n(z_j)}_i\lambda^*-W^0_i\lambda^*)(W^{\tilde{x}_n(z_l)}_i\lambda^*-W^0_i\lambda^*)\right]\nonumber\\
	&=\sum_{i=1}^{M_{k,n}}\sigma^2\mathbb{P}(D_{i-1}=0|D_0=\tilde{x}_n(z_l)-\tilde{x}_n(z_j))-\sum_{i=1}^{M_{k,n}}\sigma^2\mathbb{P}(D_{i-1}=0|D_0=\tilde{x}_n(z_j)))\nonumber\\
	&\qquad-\sum_{i=1}^{M_{k,n}}\sigma^2\mathbb{P}(D_{i-1}=0|D_0=\tilde{x}_n(z_{l}))+\sum_{i=1}^{M_{k,n}}\sigma^2\mathbb{P}(D_{i-1}=0|D_0= 0 )),.
	\end{align}
	Notice that, for any $M(x) \ge |x|^2$, we may obtain 
	\begin{equation*}
	\lim_{|x|\rightarrow \infty}\frac{1}{\log|x|}\sum_{i=1}^{M(x)}\sigma^2\left\{\mathbb{P}(D_{i-1}=0|D_0=0)-\mathbb{P}(D_{i-1}=0|D_0= x)\right\} = c
	\end{equation*}
	 in the same way as in the proof of \eqref{eq8**}. Hence
	\begin{align}\label{eq_170}
	&\lim_{n\rightarrow \infty}\frac{1}{c\mathcal{P}_n}\sum_{i=1}^{M_{k,n}}\sigma^2\left\{\mathbb{P}(D_{i-1}=0|D_0=0)-\mathbb{P}(D_{i-1}=0|D_0=\tilde{x}_n(z_{l})-\tilde{x}_n(z_{j}))\right\}\nonumber\\
	&=\lim_{n\rightarrow \infty}\frac{\log(|\tilde{x}_n(z_{l})-\tilde{x}_n(z_{j})|)}{2\log n}\nonumber \\
	&=\lim_{n\rightarrow \infty}\frac{\log[n^{\max\{|z_j(1)|,|z_j(2)|,|z_l(1)|,|z_l(2)|\}}(1+o(1))]}{2\log n} \nonumber\\
	&=\frac{\max\{|z_j(1)|,|z_j(2)|,|z_l(1)|,|z_l(2)|\}}{2}.
	\end{align}
	Then, subtracting  $\sum_{i=1}^{M_{k,n}}\sigma^2\mathbb{P}(D_{i-1}=0|D_0=0)$  from the right hand side of \eqref{eq_169}, and using \eqref{eq_170}, we find that
	\begin{align*}
		&\lim_{n\to\infty}\frac{1}{c\mathcal{P}_n}\sum_{i=1}^{M_{k,n}}\mathbb{E}\left[(W^{\tilde{x}_n(t_j)}_i\lambda^*-W^0_i\lambda^*)(W^{\tilde{x}_n(t_l)}_i\lambda^*-W^0_i\lambda^*)\right]\\
		& = -\frac{\max\{|z_j(1)|,|z_j(2)|,|z_l(1)|,|z_l(2)|\}}{2}+\frac{\max\{|z_l(1)|,|z_l(2)|\}}{2}+\frac{\max\{|z_j(1)|,|z_j(2)|\}}{2} \\
		&= \frac{\min \{\max\{|z_l(1)|,|z_l(2)|\},\max\{|z_j(1)|,|z_j(2)|\}\}}{2}. 
	\end{align*}
	Therefore,
	\begin{align*}
		&\lim_{n\to\infty}\frac{2}{c\mathcal{P}_n}\sum_{1\leq j<l\leq k}\alpha_j\alpha_l\sum_{i=1}^{M_{k,n}}\mathbb{E}\left(W^{\tilde{x}_n(t_j)}_i\lambda^*-W^0_i\lambda^*\right)\left(W^{\tilde{x}_n(t_l)}_i\lambda^*-W^0_i\lambda^*\right)\\
		& = \sum^{k}_{1\le j < l\le k}\alpha_l\alpha_j\min \{\max\{|z_l(1)|,|z_l(2)|\},\max\{|z_j(1)|,|z_j(2)|\}\},
	\end{align*} 
	and the proof of Corollary \ref{corollary_1} for $d=2$ is finished.

	\section{Proof of Proposition \ref{proposition_2} }\label{Section3}
	
	In this section, we prove Proposition \ref{proposition_2}. We follow the strategy adopted in the proof of Theorem $4.1$ in \cite{FerrariFontes}. The main difference in our case is in the following Lemma \ref{LosA}, which is analogous to Lemma 4.3 in \cite{FerrariFontes}, on the one hand, but the proof of the latter result does not apply in our more general, not necessarily nearest neighbor case.
		\begin{lemma}\label{LosA}Let us consider $d = 1,2$. Given any positive integer $K$, we have that
		\begin{equation*}
			\sum^{n}_{i=j}\left[\mathbb{P}(D_i=0|D_j=l)-\mathbb{P}(D_i=0|D_j=l')\right],
		\end{equation*}
		is uniformly bounded in $n$, $j$, $l$ and $l'$ such that $|l-l'|\leq K$.
	\end{lemma}
	
	\begin{proof}To avoid the trivial case, let us assume $l \neq l'$.  Also, let us consider $\tau$ as defined in~\eqref{tau}. 
		By the Strong Markov property, we have that
		\begin{equation*}
			\mathbb{P}\left(D_i = 0|D_0 = x\right)  =\sum_{k = 0}^i \mathbb{P}_0\left(D_{i - k} = 0\right)\mathbb{P}_0(\tau = k),
	\end{equation*}
	Hence,
	\begin{equation*}
		\sum^{n}_{i=0}\mathbb{P}(D_i=0|D_0=x)  = \sum_{i = 0}^n\sum_{k = 0}^i \mathbb{P}_0\left(D_{i - k} = 0\right)\mathbb{P}_0(\tau = k)= \sum^n_{k=0}\left(\sum^{n-k}_{i=0}\mathbb{P}_0(D_i=0)\right)\mathbb{P}_x(\tau=k).
	\end{equation*}
	for all $x\in\mathbb{Z}^d$. Then
	\begin{equation}\label{eq22**} 
		\sum^{n}_{i=0}\mathbb{P}(D_i=0|D_0=l)-\sum^{n}_{i=0}\mathbb{P}(D_i=0|D_0=l') =\sum^n_{i=0}\left(\sum^{n-i}_{k=0}\mathbb{P}_0(D_k=0)\right)\left[\mathbb{P}_{l}(\tau=i)-\mathbb{P}_{l'}(\tau=i)\right].
	\end{equation}
	Using \eqref{eq12}, \eqref{eq11**}, and \eqref{losH} in \eqref{eq22**}, and the fact that $\tau$ has the same distribution for the chain $D$ and $H$, we obtain that 
	\begin{equation*}
		\left|\sum^{n}_{i=0}\left[\mathbb{P}(D_i=0|D_0=l+(1,0))-\mathbb{P}(D_i=0|D_0=l)\right]\right|\leq C
		\left|\sum^{n}_{i=0}\left[\mathbb{P}(H_i=0|H_0=l+(1,0))-\mathbb{P}(H_i=0|H_0=l)\right]\right|
	\end{equation*}
	for some positive constant $C$. Therefore, we will prove the lemma for the homogeneous chain $H$. 
	 Let us take $L = l - l'$. Then, by the spatial homogeneity of the Markov chain $H$, we have that
		\begin{equation}\label{eq23**}
		\sum^{n}_{k=0}\mathbb{P}_0(H_k=l')=\mathbb{E}_0\left(\sum^n_{k=0}\mathbf{1}_{\{H_k=l'\}}\right)=\mathbb{E}_L\left(\sum^n_{k=0}\mathbf{1}_{\{H_k=l\}}\right) =\mathbb{E}_L\left(\sum^{n\wedge\tau-1}_{k=0}\mathbf{1}_{\{H_k=l\}}\right)+\mathbb{E}_L\left(\sum^{n}_{k=n\wedge \tau}\mathbf{1}_{\{H_k=l\}}\right).
		\end{equation}
		Observe that 
		\begin{equation}\label{eq_101}
			\mathbb{E}_L\left(\sum^{n}_{k=n\wedge \tau}\mathbf{1}_{\{H_k=l\}}\right) =\sum_{i\ge 0}\sum^{n}_{k=n \wedge i}\mathbb{P}_0(H_{k-i}=l)\mathbb{P}_L(\tau=i)=\sum_{i\ge 0}\mathbb{P}_L(\tau=i)\sum^{n-n\wedge i}_{k=0}\mathbb{P}_0(H_{k}=l).
		\end{equation}
		Using \eqref{eq_101} in \eqref{eq23**}, we obtain
		\begin{equation}\label{eq26**}
		\sum^{n}_{k=0}\{\mathbb{P}_0(H_k=l)-\mathbb{P}_0(H_k=l')\}= \sum_{i\ge 1}\mathbb{P}_L(\tau=i)\sum^{n}_{k=n-n\wedge i +1}\mathbb{P}_0(H_{k}=l) -\mathbb{E}_L\left(\sum^{n\wedge\tau-1}_{k=0}\mathbf{1}_{\{H_k=l\}}\right).
		\end{equation}
		For the first term in the right hand side of \eqref{eq26**}, observe that
		\begin{align}\label{eq_102}
		\sum_{i\ge 1}\mathbb{P}_L(\tau=i)\sum^{n}_{k=n-n\wedge i +1}\mathbb{P}_0(H_{k}=l)\nonumber &=\sum_{i = 1}^n \mathbb{P}_L(\tau=i)\sum^{n}_{k=n- i +1}\mathbb{P}_0(H_{k}=l) + \mathbb{P}_L(\tau > n)\sum^{n}_{k=1}\mathbb{P}_0(H_{k}=l) \nonumber\\
		&=\sum_{k = 1}^n\mathbb{P}_0(H_k = l)\sum_{i = n - k + 1}^{n}\mathbb{P}_L(\tau = i) + \mathbb{P}_L(\tau > n)\sum^{n}_{k=1}\mathbb{P}_0(H_{k}=l) \nonumber\\
		&= \sum_{k = 1}^n\mathbb{P}_0(H_k = l)\mathbb{P}_L(n - k + 1\le \tau \le n) + \mathbb{P}_L(\tau > n)\sum^{n}_{k=1}\mathbb{P}_0(H_{k}=l) \nonumber\\
		&=\sum_{k = 1}^{n}\mathbb{P}_0(H_k = l)\mathbb{P}_L(\tau\ge n - k + 1) \nonumber\\
		&= \sum_{k = 0}^{n - 1}\mathbb{P}_0(H_{n - k} = l)\mathbb{P}_L(\tau\ge k + 1).
		\end{align}
		
		Using \eqref{eq_102} in \eqref{eq26**}, we find that
		\begin{equation}\label{eq_103}
			\sum^{n}_{k=0}\left\{\mathbb{P}_0(H_k=l)-\mathbb{P}_0(H_k=l')\right\} = \sum_{k = 0}^{n - 1}\mathbb{P}_0(H_{n - k} = l)\mathbb{P}_L(\tau\ge k + 1) - \mathbb{E}_L\left(\sum^{n\wedge\tau-1}_{k=0}\mathbf{1}_{\{H_k=l\}}\right) .
		\end{equation}
	
		In the rest of the proof, we deal separately with the two terms in the right hand side of \eqref{eq_103}. For the first term, observe that by \eqref{eq10}, P4 in \cite{Spitzer} (page 382) and the corollary of Theorem 1 in \cite{rrowe} imply that
		\begin{equation}\label{eq29**}
		\sum^{n-1}_{k=0}\mathbb{P}_0(H_{n-k}=l)\mathbb{P}_L( \tau\geq k+1)\leq\left\lbrace
		\begin{array}{cc}
		C_1\sum^{n-1}_{k=1}\frac{1}{\sqrt{n-k}}\frac{a(L)}{\sqrt{k}} ,& \text{ in }d=1,\\
		C_2\sum^{n-1}_{k=1}\frac{1}{n-k}\frac{\log |L|}{\log k} ,& \text{ in }d=2.\\
		\end{array}\right.
		\end{equation}
		Notice that
		$$
		\sum^{n-1}_{i=1}\frac{1}{\sqrt{n-i}}\frac{1}{\sqrt{i}}\underset{n\rightarrow \infty}{\longrightarrow} \int^1_0\frac{1}{\sqrt{(1-y)y}}dy.
		$$
Since this integral is finite and $|L|\leq K$,  there exists a uniform upper bound for the left hand side expression in \eqref{eq29**} when $d=1$. 

For the bound in $d=2$, notice that
		\begin{align*}
		\sum^{n-1}_{i=1}\frac{1}{n-i}\frac{1}{\log i}&\leq\sum^{(1-\epsilon)n}_{i=1}\frac{1}{n-i}\frac{1}{\log i}+\sum^{n}_{i=(1-\epsilon)n}\frac{1}{n-i}\frac{1}{\log i}\\
		&\leq \frac{(1-\epsilon)n}{n} + \frac{1}{\log((1-\epsilon)n)}\sum^{\epsilon n}_{i=1}\frac{1}{\epsilon}\\
		&\leq(1-\epsilon)+\frac{\log(\epsilon n)}{\log((1-\epsilon)n)}\\
		& \leq C_2,
		\end{align*}
		for some positive constant $C_2$ and for all $n\geq 1$, and we have that the left hand side of \eqref{eq29**} is bounded for $d=2$ also.
		
		Back to the second term in \eqref{eq_103}, we observe that it is bounded above by
		\begin{equation}\label{eq30**}
		\mathbb{E}_L\left(\sum^{\tau}_{k=0}\mathbf{1}_{\{H_k=l\}}\right)=\tilde{g}_{\{0\}}(L,l),
		\end{equation}
		using the notation of $(1.14)$ in \cite{KestenSpitzer}. It is proved in \cite{KestenSpitzer} that 
\begin{equation}\label{KestenSpitzer1}
\underset{l\rightarrow \infty}{\lim}\tilde{g}_{\{0\}}(L,l)<\infty, \text{ when }d=2
\end{equation}		
	and
	\begin{equation}\label{KestenSpitzer2}
\underset{l\rightarrow \infty}{\lim}\frac{1}{2}(\tilde{g}_{\{0\}}(L,l)+\tilde{g}_{\{0\}}(L,-l))<\infty, \text{ when }d=.
\end{equation}	
 	\eqref{KestenSpitzer1} and \eqref{KestenSpitzer2} correspond to $(1.16)$ and $(1.17)$ in \cite{KestenSpitzer}, respectively. 
 	By \eqref{KestenSpitzer1} and \eqref{KestenSpitzer2} we have that the second term in the right hand side of \eqref{eq30**} 
 	is uniformly bounded for $|L|\le K$. 
	
	\end{proof}
	
\begin{proof}[Proof of Proposition \ref{proposition_2}:] To prove Proposition \ref{proposition_2}	 we  will apply Corollary $3.1$ in \cite{Hall&Hyde} for the variables
	$$
	X_{x,i}=\frac{W^{x}_i\lambda^*-W^0_i\lambda^*}{\sqrt{\mathcal{P}_x}}, \, 1\leq i\leq A|x|^2 .
	$$
	In  Lemma $2.2$ in \cite{FerrariFontes} is proved that
	$$
	\left \lbrace\sum^{n}_{i=1}W^{x}_i\lambda^* -n\mu\lambda^* \right \rbrace_{n\geq 1},
	$$
	is a martingale with respect to the filtration $\mathcal{F}_n$ defined in Section \ref{Notation} . Observe that a linear combination of martingales is also a martingale. Therefore,  $X_{x,i}$ is the increment of a nested mean zero and square-integrable martingale with respect to the nested filtration  $\mathcal{F}_{x,i}=\mathcal{F}_i$. Hence, by Corollary $3.1$ in \cite{Hall&Hyde}, if we check
	\begin{align}
	&\underset{1 \le i\leq A|x|^2}{\max} \left|X_{x,i}\right|\overset{p}{\longrightarrow}0,\label{Condition1}\\
	&\mathbb{E}\left(\underset{1\le i\leq A|x|^2}{\max}X_{x,i}\right)\text{ is bounded in }x,\label{Condition2}
	\end{align}
	and
	\begin{align}\label{Condition3}
	\sum^{A|x|^2}_{i=1}\mathbb{E}\left(X^2_{x,i}| \mathcal{F}_{x,i-1}\right)\overset{p}{\longrightarrow} h(A)\, \text{ as } |x|\to\infty,
	\end{align}
	 we obtain
	
	\begin{equation}\label{eq16}
	\sum^{A|x|^2}_{i=1}X_{x,i}\overset{d}{\longrightarrow} Z \,  \text{ as } |x| \to\infty,
	\end{equation}
	where $Z$ is a mean zero Gaussian r.v. with variance $h(A)$.
	
	Notice that conditions \eqref{Condition1} and \eqref{Condition2}  follow straightforwardly from the fact that 
	$u_n(i,i+\cdot)$ have bounded support. 
	It remains to prove \eqref{Condition3}. To do this, observe that by the definition of $h(A)$ and \eqref{eq_107}, we have that 
	\begin{equation}\label{h_2}
		h(A) = \lim_{|x|\to\infty} \frac{1}{\mathcal{P}_x}\mathbb{E}\sum^{A|x|^2}_{i=1}\left([W^x_i-W^0_i]\lambda^*\right)^2, \, \text{ for  all } A > 0.
	\end{equation}
	 
	Given \eqref{h_2}, to obtain \eqref{Condition3}, it is enough to prove that 
	\begin{equation}\label{eq17}
		\frac{1}{\mathcal{P}_x}\left\lbrace\sum^{A|x|^2}_{i=1}\mathbb{E}\left([W^x_i\lambda^*-W^0_i\lambda^*]^2|\mathcal{F}_{i-1}\right)-\mathbb{E}\left(W^x_i\lambda^*-W^0_i\lambda^*\right)^2\right\rbrace\overset{p}{\longrightarrow}0, \, \text{ as } |x|\to\infty.
	\end{equation}
	Let us denote the range of $u_n(i,i+\cdot)$ by $R$.
		We consider $\tilde{D}_n^x = \tilde{Y}_n^x - \hat{Y}_n^0$, with $\hat{Y}_n^0$ an independent copy of $\tilde{Y}_n^x$ given $\mathcal{F}_n$. Using the same calculations used to get $(5.7)$ in \cite{FerrariFontes}, we can obtain
		\begin{equation}\label{eq_159}
		\mathbb{E}\left((W_i^z-\mu)(W_i^y-\mu)\Big|\mathcal{F}_{i-1}\right) =\sigma^2 \mathbb{P}\left(\tilde{D}_{i-1}^{z-y} = 0\big|\mathcal{F}_{i-1}\right), \, \text{ for } i\ge 1, \text{ and } y,z\in\mathbb{Z}^d,
		\end{equation}
		which implies
		\begin{align*}
		&\mathbb{E}[(W^x_i-W^0_i)^2|\mathcal{F}_{i-1}]-\mathbb{E}(W^x_i-W^0_i)^2\\&=2\sigma^2\left[\mathbb{P}(\tilde{D}^0_{i-1}=0|\mathcal{F}_{i-1})-\mathbb{P}(\tilde{D}^0_{i-1}=0)\right] - 2\sigma^2\left[\mathbb{P}(\tilde{D}^x_{i-1}=0|\mathcal{F}_{i-1})-\mathbb{P}(\tilde{D}^x_{i-1}=0)\right].
		\end{align*}
		Therefore, it is sufficient to prove 
		\begin{equation}\label{eq19}
		\frac{1}{\mathcal{P}_x}\sum^{A|x|^2}_{i=1}\left(\mathbb{P}(\tilde{D}^x_{i-1}=0|\mathcal{F}_{i-1})-\mathbb{P}(\tilde{D}^x_{i-1}=0)\right)\overset{p}{\longrightarrow}0, \text{ as } |x|\to \infty,
		\end{equation}
		and
		\begin{equation}\label{eq20}
		\frac{1}{\mathcal{P}_x}\sum^{A|x|^2}_{i=1}\left(\mathbb{P}(\tilde{D}^0_{i-1}=0|\mathcal{F}_{i-1})-\mathbb{P}(\tilde{D}^0_{i-1}=0)\right)\overset{p}{\longrightarrow}0, \text{ as } |x|\to\infty.
		\end{equation}
		The computations to obtain \eqref{eq19} and \eqref{eq20}  are very similar to the ones in the proof of Theorem $4.1$ in \cite{FerrariFontes}. To avoid repetition and for the sake of completeness, we give them 
		in Appendix \ref{appendix_1}. To conclude for now, let us point out that Lemma \ref{LosA} is used in the proof of \eqref{eq19} and \eqref{eq20}.

\end{proof}

\section{Corollary of Proposition \ref{proposition_2}}\label{sec_corollary_2}
	As mentioned in Section \ref{sec_corollary_1}, Cram\'er-Wold theorem allows us to deduce the convergence of a sequence of vectors through the study of linear combinations of its components. Hence, to obtain Theorem \ref{Thm_2}, we state and prove the following result 
	in this section. 

	\begin{corollary}  \label{corollary_2} Given $k\ge 1$, let $\bar{\alpha}=(\alpha_1,\dots,\alpha_k)\in\mathbb{R}^k$. Let us also consider  $c$  and $\mathcal{P}_n$ as in Propostion \ref{proposition_1}. 
	\begin{enumerate}
		\item[(i)] In case $d = 1$, let $\bar{t} = (t_1,\dots,t_k)\in\mathbb{R}^k$ with $0=t_0 < t_1,\dots < t_k$. Then we have
		\begin{equation*}
		\frac{1}{\sqrt{c\mathcal{P}_n}}\sum_{j=1}^k\sum_{i=1}^{An^2}\alpha_j\left(W_i^{\lfloor nt_j\rfloor} - W_i^0\right)\lambda \overset{d}{\underset{n\rightarrow \infty}{\longrightarrow}} \mathcal{N}(0,g(A,\bar{t},\bar{\alpha})),
		\end{equation*}
		where the function $g$ is as in Corollary \ref{corollary_1} for dimension one.
		\item[(ii)] In case $d = 2$, let $\bar{z}=(z_1,\dots,z_k)\in \mathbb{Z}^{2k}$. Then we have
		\begin{equation*}
		\frac{1}{\sqrt{c\mathcal{P}_n}}\sum_{j=1}^k\sum_{i=1}^{M_{k,n}}\alpha_j\left(W_i^{\tilde{x}_n(z_j)} - W_i^0\right)\lambda^* \overset{d}{\underset{n\rightarrow \infty}{\longrightarrow}} \mathcal{N}(0,g(\bar{z},\bar{\alpha})),
		\end{equation*}
		where $M_{k,n}$ and $g$ are as in Corollary \ref{corollary_1} for dimension two.
	\end{enumerate} 
\end{corollary}
\begin{proof} To avoid being repetitive, we only write the proof for dimension one, but the reader can check that the same ideas apply to dimension 2.
	
As in the proof of Proposition \ref{proposition_2}  we  will apply Corollary $3.1$ in \cite{Hall&Hyde} but now for the variables
$$
X^k_{n,i}=\frac{1}{\sqrt{c\mathcal{P}_n}}\sum_{j=1}^{k}\alpha_j\left(W_i^{\lfloor nt_j\rfloor} - W_i^0 \right)\lambda, \, 1\leq i\leq An^2 .
$$
Again conditions \eqref{Condition1} and \eqref{Condition2}  follow readily from the fact that we are assuming finite support for $u_n(i,i+\cdot)$.  

It remains to argue \eqref{Condition3}. In order to that, let us first notice that equation \eqref{eq_107} implies (after straightforward computations) that
\begin{equation}\label{eq_150}
	\sum_{i=1}^{An^2}\mathbb{E}\left[\left(\sum_{j=1}^{k}\alpha_j\left[W_i^{\lfloor nt_j\rfloor} - W_i^0\right]\lambda\right)^2\right] = \mathbb{E}\left[\left(\sum_{j=1}^k\alpha_j\sum_{i=1}^{An^2}\left[W_i^{\lfloor nt_j\rfloor} - W_i^0\right]\lambda\right)^2\right].
\end{equation}
Hence, by Corollary \ref{corollary_1} we have 
\begin{equation}\label{eq_151}
	\lim_{n\to\infty}\sum_{i=1}^{An^2}\mathbb{E}\left[\left(X_{n,i}^k\right)^2\right] = g(A,\bar{t},\bar{\alpha}).
\end{equation}
Because 
\begin{align*}
	(X_{n,i}^k)^2 =  \frac{1}{c\mathcal{P}_n}\sum_{j=1}^k\left(W_i^{\lfloor nt_j\rfloor}\lambda - W_i^0\lambda \right)^2 + 2\sum_{1\le j < l\le k}\alpha_j\alpha_l\left(W_i^{\lfloor nt_j\rfloor}\lambda - W_i^0\lambda \right)\left(W_i^{\lfloor nt_l\rfloor}\lambda - W_i^0\lambda \right),
\end{align*}
and, from \eqref{eq19},
\begin{equation*}
	\sum_{j=1}^k\frac{1}{c\mathcal{P}_n}\sum_{i = 1}^{An^2}\left\{ \mathbb{E}\left(\left[ W_i^{\lfloor nt_j\rfloor}\lambda - W_i^0\lambda \right]^2\Big|\mathcal{F}_{i-1}\right) -\mathbb{E}\left(\left[ W_i^{\lfloor nt_j\rfloor}\lambda n - W_i^0\lambda \right]^2\right)\right\} \overset{p}{\longrightarrow}0 \, \text{ as } n\to\infty,
\end{equation*}
in order to obtain Condition \eqref{Condition3}  (and conclude the proof), it is enough to show that
\begin{equation}\label{eq_152}
	\frac{1}{c\mathcal{P}_n}\sum_{i=1}^{An^2}\left\{\mathbb{E}\left((W_i^{\lfloor nt_j\rfloor}\lambda - W_i^0\lambda )(W_i^{\lfloor nt_l\rfloor}\lambda - W_i^0\lambda )\big|\mathcal{F}_{i-1}\right) - \mathbb{E}(W_i^{\lfloor nt_j\rfloor}\lambda - W_i^0\lambda)(W_i^{\lfloor nt_l\rfloor}\lambda - W_i^0\lambda)\right\}
\end{equation}
converges in probability to zero as $n\to\infty$, for any  $1\le j<l\le k$. 
Set
\begin{equation*}
	S_n^{y,z} := \frac{\sigma^2}{c\mathcal{P}_n}\sum_{i=1}^{An^2}\left\{\mathbb{P}\left(\tilde{D}_{i-1}^{z-y} = 0\big|\mathcal{F}_{i-1}\right) -\mathbb{P}\left(\tilde{D}_{i-1}^{z-y} = 0\right)\right\}.
\end{equation*}
By the the same argument for \eqref{eq_161}, and using \eqref{eq_159}, we get  that \eqref{eq_152} equals
\begin{equation}\label{eq_153}
	S_n^{\lfloor nt_j\rfloor,\lfloor nt_l\rfloor} - S_n^{\lfloor nt_j\rfloor,0\rfloor} - S_n^{\lfloor nt_l\rfloor, 0} + S_n^{0,0}.
\end{equation}
 Now, by \eqref{eq19} and \eqref{eq20}, we have that each term in equation \eqref{eq_153} goes to zero in probility as $n$ goes to infinity, and this establishes  \eqref{Condition3}.
\end{proof}

\section{Proof of Theorem \ref{clt}}\label{sec_clt}\label{sec_proof_Thm_1}
Now we have all the ingredients to prove Theorem \ref{clt}.
\begin{proof}[Proof of Theorem \ref{clt}]
	By \eqref{series_representation} we have that
	\begin{equation}\label{eq_127}
	\frac{\widehat{X}_x - x\lambda^*}{\sqrt{\mathcal{P}_x}} \overset{d}{=} \frac{1}{\sqrt{\mathcal{P}_x}}\sum_{i = 0}^{A|x|^2}\left(W_i^x - W_i^0\right)\lambda^* + \frac{1}{\sqrt{\mathcal{P}_x}}\sum^{\infty}_{i = A|x|^2+1}\left(W_i^x - W_i^0\right)\lambda^* \, , \text{ for any } A \ge 1 .
	\end{equation}
By equation  $(5.14)$ in \cite{FerrariFontes} and \textbf{P28.4}\footnote{To apply \textbf{28.4} is necessary  to note that the function $a$ is even because the chain $H$ is symmetric.}(page 345) in \cite{Spitzer}, for $d = 1$ and equation \eqref{eq_140} for $d = 2$, we have that
\begin{align}\label{eq_126}
	\lim_{|x|\to\infty}\frac{\mathbb{E}\left(\widehat{X}_{\infty}(x) - x\lambda^*\right)^2}{\mathcal{P}_x} =c .
\end{align}
	By \eqref{series_representation} and \eqref{eq_107}  we obtain
\begin{equation}\label{eq_125}
\frac{\mathbb{E}\left(\widehat{X}_{\infty}(x) - x\lambda^*\right)^2}{\mathcal{P}_x}=\frac{1}{\mathcal{P}_x}\mathbb{E}\left(\sum^{A|x|^2}_{i=1}(W^x_i-W^0_i)\lambda^*\right)^2+\frac{1}{\mathcal{P}_x}\mathbb{E}\left(\sum^{\infty}_{i=A|x|^2+1}(W^x_i-W^0_i)\lambda^*\right)^2.
\end{equation}	
Hence by \eqref{eq_126}, \eqref{eq_125}, and Proposition \ref{proposition_1} we obtain
	\begin{equation}\label{eq_106}
	\lim_{A\to\infty}\lim_{x\to \infty}\mathbb{E}\left[\left(\frac{1}{\sqrt{\mathcal{P}_x}}\sum^{\infty}_{i = A|x|^2+1}\left(W_i^x - W_i^0\right)\lambda^*\right)^2\right] = 0 .
	\end{equation}
	Only in the fallowing calculations and to simplify notations  we denote the left member  in \eqref{eq_127} by $X_x$, and the first and second term by $Y^A_x$ and $Z^A_x$, respectively. 
	
	Take $f$ a uniformly continuous and bounded function and we denote by $M_f$ the bound of $f$. Also, for a fix $\epsilon>0$ we take $\delta$ such that $|f(z)-f(z+y)|\leq \epsilon/6$ for all $z$ and all $y$ such that $|y|\leq \delta$.  In addition, we take $A$ large enough such that 
	\begin{equation}\label{eq_128}
	\left \lvert\int f(Y^A)d\mathbb{P}-\int f(Y)d\mathbb{P}\right \rvert\leq \frac{\epsilon}{3},
	\end{equation}
	where $Y^A$ and $Y$ are distributed $\mathcal{N}(0,h(A))$ and $\mathcal{N}(0,c)$, respectively. Now, for this $A$ we take $|x|$ large enough such that
	\begin{equation}\label{eq_130}
	\mathbb{P}(|Z^A_x|>\delta)\leq \frac{\epsilon}{12 M_f}\text{ and }\left\lvert\int f(Y^A_x)d\mathbb{P}-\int f(Y^A)d\mathbb{P}\right\rvert \leq \frac{\epsilon}{3}.
	\end{equation}
	The first part of \eqref{eq_130} is a consequence of \eqref{eq_106} and the second is the result in Proposition \ref{proposition_2}. Observe that
	\begin{align}\label{eq_129}
	\left\lvert \int f(X_x)d\mathbb{P}-\int f(Y)d\mathbb{P}\right\rvert &\leq \left\lvert\int f(X_x)d\mathbb{P}-\int f(Y^A_x)d\mathbb{P}\right\rvert+\left\lvert\int f(Y^A_x)d\mathbb{P}-\int f(Y^A)d\mathbb{P}\right\rvert\nonumber\\
	&+ \left\lvert\int f(Y^A)d\mathbb{P}-\int f(Y)d\mathbb{P}\right\rvert.
	\end{align}
	By \eqref{eq_130} and \eqref{eq_128} the second and third term in the right member of \eqref{eq_129} are both less than $\epsilon/3$. For the first term, using \eqref{eq_130} and our choice of $\delta$ we have that
	\begin{align*}
	\left \lvert \int f(X_x)d\mathbb{P}-\int f(Y^A_x)d\mathbb{P}\right\rvert \leq  \frac{2\epsilon}{12 M_f} +\left \lvert \int_{|Z^A_x|<\delta}(f(X_x)-f(Y^A_x)) d \mathbb{P}\right \rvert \leq \frac{\epsilon}{3}.
	\end{align*}
	Therefore, for every $\epsilon>0$  and $|x|$ large enough, we have concluded that
	\begin{align*}
	\left\lvert \int f(X_x)d\mathbb{P}-\int f(Y)d\mathbb{P}\right\rvert\leq \epsilon,
	\end{align*}
	obtaining the desired converge in distribution.
	\end{proof}

	\section{Proof of Theorem \ref{Thm_2}}\label{sec_proof_of_Thm_2}
	
	Although, the proof of Theorem \ref{Thm_2} is very similar for dimension one and two, it has some specific  calculations that are different for each case. Therefore, we start with the proof in dimension one, and after that, we prove the core different part for dimension two. 
	\begin{proof}[Proof for $d=1$] By Cram\'er-Wold theorem, the convergence in \eqref{convergence_thm_2_dim_1} is equivalent to 
		\begin{equation}\label{eq_155}
		\sum_{j=1}^k\alpha_j X_n(t_j)\overset{d}{\underset{n\rightarrow \infty}{\longrightarrow}} \sum_{j=1}^k\alpha_jB(t_j), \, \text{ for any } (\alpha_1,\dots,\alpha_k)\in\mathbb{R}^k.
		\end{equation}
		Hence, we will prove \eqref{eq_155}. To do this, first notice that $\sum_{j=1}^k\alpha_jB(t_j)$ is a Gaussin random variable with mean zero and variance $\sum_{j=1}^k\alpha_j^2t_j + 2\sum_{1\le j < l \le k}\alpha_j\alpha_lt_j$. 
		As in the proof of Theorem \ref{clt}, we use that 
		\begin{align}\label{eq_192}
		\sum_{j=1}^k\alpha_j X_n(t_j) &\overset{d}{=} \frac{1}{\sqrt{c\mathcal{P}_n}}\sum_{j=1}^k\sum_{i=1}^{\infty}\alpha_j\left(W_i^{\lfloor nt_j\rfloor} - W_i^0\right)\lambda \nonumber\\
		&= \frac{1}{\sqrt{c\mathcal{P}_n}}\sum_{j=1}^k\sum_{i=1}^{An^2}\alpha_j\left(W_i^{\lfloor nt_j\rfloor} - W_i^0\right)\lambda + \frac{1}{\sqrt{c\mathcal{P}_n}}\sum_{j=1}^k\sum_{i >An^2}\alpha_j\left(W_i^{\lfloor nt_j\rfloor} - W_i^0\right)\lambda .
		\end{align}
		
		To deal with the first sum in the right hand side of \eqref{eq_192}, we use Corollary \ref{corollary_1} and Corollary \ref{corollary_2}. It will remain to compute the limit as $n\to\infty$ of the second moment of  the left hand side of \eqref{eq_192}. If this limit is equal to the one in Corollary \ref{corollary_1}, then the second sum in right hand side of \eqref{eq_192} will be small in probability when $A$ is large, and  this allows us to follow the proof of Theorem \ref{clt} in a straightforward fashion to obtain \eqref{eq_155}. Summing up, it is enough to show that
		\begin{equation}\label{eq_156}
		\lim_{n\to\infty}\frac{1}{c\mathcal{P}_n}\mathbb{E}\left[\left(\sum_{j=1}^k\alpha_j\sum_{i=1}^{\infty}\left[W_i^{\lfloor nt_j\rfloor} - W_i^0\right]\lambda\right)^2\right] = \sum_{j=1}^k\alpha_j^2t_j + 2\sum_{1\le j < l <\le k}\alpha_j\alpha_lt_j.
		\end{equation}

		Expanding the square inside the mean in the left member of \eqref{eq_156} we have that
		\begin{align}\label{eq_194}
		&\frac{1}{c\mathcal{P}_n}\mathbb{E}\left[\left(\sum_{j=1}^k\alpha_j\sum_{i=1}^{\infty}\left[W_i^{\lfloor nt_j\rfloor} - W_i^0\right]\lambda\right)^2\right]= \sum_{j=1}^k\alpha_j^2\frac{1}{c\mathcal{P}_n}\mathbb{E}\left[\left(\sum_{i=1}^{\infty}\left[W_i^{\lfloor nt_j\rfloor} - W_i^0\right]\lambda\right)^2\right]\nonumber\\
		&\hspace{2.2cm}+ 2\sum_{1\le j<l\le k}\frac{1}{c\mathcal{P}_n}\mathbb{E}\left[\left(\sum_{i=1}^{\infty}\left[W_i^{\lfloor nt_j\rfloor} - W_i^0\right]\lambda\right)\left(\sum_{m=1}^{\infty}\left[W_m^{\lfloor nt_l\rfloor} - W_m^0\right]\lambda\right)\right].
		\end{align}
		
		Now, we deal with the first sum in the right hand side of \eqref{eq_194}. Expanding the square inside the mean, and using \eqref{eq_107}, we find that 
		\begin{equation}\label{eq_195}
		\sum_{j=1}^k\alpha_j^2\frac{1}{c\mathcal{P}_n}\mathbb{E}\left[\left(\sum_{i=1}^{\infty}\left[W_i^{\lfloor nt_j\rfloor} - W_i^0\right]\lambda\right)^2\right]
		=\sum_{j=1}^k\alpha_j^2\frac{1}{c\mathcal{P}_n}\sum_{i=1}^{\infty}\mathbb{E}\left[\left(W_i^{\lfloor nt_j\rfloor} - W_i^{0}\right)\lambda\right]^2.
		\end{equation}
		
		Equation $5.14$  in \cite{FerrariFontes} implies that the limit of the sum in the right hand side of \eqref{eq_195} equals  $\sum_{j=1}^k\alpha_j^2t_j$. Therefore, to get the convergence in \eqref{eq_156}, it is enough to prove that
		\begin{equation}\label{eq_158}
		\lim_{n\to\infty}\frac{1}{c\mathcal{P}_n}\mathbb{E}\left[\left(\sum_{i=1}^{\infty}\left[W_i^{\lfloor nt_j\rfloor} - W_i^0\right]\lambda\right)\left(\sum_{m=1}^{\infty}\left[W_m^{\lfloor nt_l\rfloor} - W_m^0\right]\lambda\right)\right] = t_j, \, \text{ for any  } 1\le j<l\le k.
		\end{equation}
		
		Again, multiplying the sums inside the mean in \eqref{eq_158} and using \eqref{eq_107} we find that
		\begin{equation*}
		\mathbb{E}\left[\left(\sum_{i=1}^{\infty}\left[W_i^{\lfloor nt_j\rfloor} - W_i^0\right]\lambda\right)\left(\sum_{m=1}^{\infty}\left[W_m^{\lfloor nt_l\rfloor} - W_m^0\right]\lambda\right)\right] = \sum_{i=1}^{\infty}\mathbb{E}\left[\left(W_i^{\lfloor nt_j\rfloor}\lambda - W_i^0\lambda\right)\left(W_i^{\lfloor nt_l\rfloor}\lambda - W_i^0\lambda\right)\right].
		\end{equation*}
		
		Hence, what we need to prove is
		\begin{equation}\label{eq_164}
		\lim_{n\to\infty}\frac{1}{c\mathcal{P}_n}\sum_{i=1}^{\infty}\mathbb{E}\left[\left(W_i^{\lfloor nt_j\rfloor}\lambda - W_i^0\lambda\right)\left(W_i^{\lfloor nt_l\rfloor}\lambda - W_i^0\lambda\right)\right] = t_j,\, \text{ for any } 1\le j<l\le k.
		\end{equation}
		
		All the calculations  we did to obtain \eqref{eq_191} are valid for the infinite sums. Therefore, 
		\begin{align}\label{eq_197}
		&\lim_{n\to\infty}\frac{1}{c\mathcal{P}_n}\sum_{i=1}^{\infty}\mathbb{E}\left[\left(W_i^{\lfloor nt_j\rfloor}\lambda - W_i^0\lambda\right)\left(W_i^{\lfloor nt_l\rfloor}\lambda - W_i^0\lambda\right)\right] \nonumber\\
		& =-\lim_{n\to\infty}\frac{\sigma^2}{c\mathcal{P}_n}\sum_{i=1}^{\infty}\left\{\mathbb{P}\left(D_{i-1} = 0|D_0 =0\right) -\mathbb{P}\left(D_{i-1} = 0|D_0 =  \lfloor nt_l\rfloor -\lfloor nt_j\rfloor\right)\right\}\nonumber\\
		&+\lim_{n\to\infty}\frac{\sigma^2}{c\mathcal{P}_n}\sum_{i=1}^{\infty}\left\{\mathbb{P}\left(D_{i-1} = 0|D_0 = 0\right)- \mathbb{P}\left(D_{i-1} = 0|D_0 =   \lfloor nt_j \rfloor\right)\right\} \nonumber\\
		&+\lim_{n\to\infty}\frac{\sigma^2}{c\mathcal{P}_n}\sum_{i=1}^{\infty}\left\{\mathbb{P}\left(D_{i-1} = 0|D_0 =0\right)  - \mathbb{P}\left(D_{i-1} = 0|D_0 =\lfloor nt_l\rfloor\right)\right\}.
		\end{align}
		By $(5.14)$ in \cite{FerrariFontes} and \textbf{P28.4} (page 345) in \cite{Spitzer},	for $s<t$ we have that 
		\begin{align*}
		&\lim_{n\to\infty}\frac{2\sigma^2}{\mathcal{P}_n}\sum_{i=1}^{\infty}\left\{\mathbb{P}\left(D_{i-1} = 0|D_0 =0\right) - \mathbb{P}\left(D_{i-1} = 0|D_0 = \lfloor nt \rfloor-\lfloor ns\rfloor \right)\right\}\\
		&=\lim_{n\to\infty}\frac{\lfloor nt \rfloor-\lfloor ns\rfloor}{\mathcal{P}_n}\frac{2\sigma^2}{\lfloor nt \rfloor-\lfloor ns\rfloor}\sum_{i=1}^{\infty}\left\{\mathbb{P}\left(D_{i-1} = 0|D_0 =0\right) - \mathbb{P}\left(D_{i-1} = 0|D_0 = \lfloor nt \rfloor-\lfloor ns\rfloor \right)\right\}=(t-s)c.
		\end{align*}
		Hence, the limit in the left hand side of \eqref{eq_197} is equal to 
		\begin{equation*} 
		-\frac{1}{2}(t_l-t_j)+\frac{1}{2}t_j+\frac{1}{2}t_l = t_j,
		\end{equation*}
		and the proofs of \eqref{eq_164} and of Theorem \ref{Thm_2} for $d=1$ is concluded.
	\end{proof}
	\begin{proof}[Proof for $d=2$]
		As in dimension one,  by Cram\'er-Wold theorem, the convergence in \eqref{convergence_thm_2_dim_2} is equivalent to
		\begin{equation}\label{eq_155*}
		\sum_{j=1}^k\alpha_j X_n(z_j)\overset{d}{\underset{n\rightarrow \infty}{\longrightarrow}} \sum_{j=1}^k\alpha_jZ_j, \, \text{ for any } (\alpha_1,\dots,\alpha_k)\in\mathbb{R}^k.
		\end{equation}
		Again, we split the infinite sum as we did for dimension one, but this time we do not split  at $n^2$, but at $M_{k,n}$. More precisely,
		\begin{align}\label{eq_200}
		&\sum_{j=1}^k\alpha_j X_n(z_j) \nonumber\\
		&\qquad\overset{d}{=} \frac{1}{\sqrt{c\mathcal{P}_n}}\sum_{j=1}^k\sum_{i=1}^{M_{k,n}}\alpha_j\left(W_i^{\tilde{x}_n(z_j)} - W_i^0\right)\lambda^* + \frac{1}{\sqrt{c\mathcal{P}_n}}\sum_{j=1}^k\sum_{i >M_{k,n}}\alpha_j\left(W_i^{\tilde{x}_n(z_j)} - W_i^0\right)\lambda^*.
		\end{align}
		Similarly as in the proof for dimension one, for the first sum in the right hand side of \eqref{eq_200}, we use Corollary \ref{corollary_2}. Following  the same arguments of the proof of Theorem \ref{clt}, it only remains to state that the second moment of the left hand side of \eqref{eq_200} converges to the same limit as in Corollary \ref{corollary_1}. Therefore, we need to prove that 
		\begin{align}\label{eq_165*}
		&\lim_{n\to\infty}\frac{1}{c\mathcal{P}_n}\mathbb{E}\left[\left(\sum_{j=1}^k\alpha_j\sum_{i=1}^{\infty}\left[W_i^{\tilde{x}_n(z_j)} - W_i^0\right]\lambda^*\right)^2\right]\nonumber\\
		& = \sum^k_{j=1}\alpha^2_j\max\{|z_j(1)|,|z_j(2)|\}+\sum^{k}_{1\le j < l\le k}\alpha_l\alpha_j\min \{\max\{|z_l(1)|,|z_l(2)|\},\max\{|z_j(1)|,|z_j(2)|\}\}.
		\end{align}
		All the arguments given in the proof of Corollary \ref{corollary_1} are also valid for the infinite sum, and \eqref{eq_165*} follows.  
	\end{proof}

\appendix

\section{Finishing the proof of Proposition \ref{proposition_2}}\label{appendix_1}

Since the arguments for \eqref{eq19} and \eqref{eq20} are entirely similar,  we give only the proof of \eqref{eq19}. 
\begin{proof}[Proof of \eqref{eq19} ]
 
Recycling  the arguments given in the equations $(4.3)$, $(4.4)$ and $(4.5)$ in \cite{FerrariFontes}, we could write the variance of the sum at  \eqref{eq19} as
	\begin{equation}\label{eq21}
	\sum^{A|x|^2}_{j=1}\mathbb{E}\left( \sum^{A|x|^2-1}_{i=j}\left\{\mathbb{P}(\tilde{D}^x_i=0|\mathcal{F}_j)-\mathbb{P}(\tilde{D}^x_i=0|\mathcal{F}_{j-1})\right\}\right)^2.
	\end{equation}
	Conditioning in $\tilde{D}^x_j$ we have that
	\begin{equation}\label{eq_109}
	\mathbb{P}(\tilde{D}^x_i=0|\mathcal{F}_j)-\mathbb{P}(\tilde{D}^x_i=0|\mathcal{F}_{j-1}) =\sum_{k}\mathbb{P}(\tilde{D}^x_i=0|\tilde{D}^x_j=k)\left[\mathbb{P}(\tilde{D}^x_j=k|\mathcal{F}_{j})-\mathbb{P}(\tilde{D}^x_j=k|\mathcal{F}_{j-1})\right].
	\end{equation}
	Now, conditioning on $\tilde{Y}^x_{j-1}$ and $\hat{Y}^0_{j-1}$ we have that $\mathbb{P}(\tilde{D}^x_j=k|\mathcal{F}_{j})-\mathbb{P}(\tilde{D}^x_j=k|\mathcal{F}_{j-1})$ is equals to
	\begin{align}\label{eqalgo}
	\sum_{l,l'}\left[\mathbb{P}(\tilde{D}^x_j=k|\tilde{Y}^x_{j-1}=l'+l,\hat{Y}^0_{j-1}=l',\mathcal{F}_j)-\mathbb{P}(\tilde{D}^x_j=k|\tilde{Y}^x_{j-1}=l'+l,\hat{Y}^0_{j-1}=l')\right]\times \mathbb{P}(\tilde{Y}^x_{j-1}=l'+l,\hat{Y}^0_{j-1}=l'|\mathcal{F}_{j-1}) .
	\end{align}
	In the rest of the proof, we consider $d=2$ and the proof for $d=1$ is similar. Back to equation \eqref{eqalgo}, observe that  we have the following relation between $k$ and $l$
	\begin{equation*}
		k = l + (\tilde{Y}^x_j - \tilde{Y}^x_{j - 1}) - (\hat{Y}^0_j - \hat{Y}^0_{j - 1}).
	\end{equation*}
	Hence, we could write $k=l+b_1 - b_2$  for  $b_1,\,b_2\in \mathbb{Z}^2$ with $\max\{|b_1|,|b_2|\}  \le K$, where $K$ is the range of the measure $u_0$.
 The finite support assumption is only required when  we use Lemma \ref{LosA}. We denote by $V$ a $K$-neighborhood of zero.
	  Also, we  use the following notation
	  	\begin{align*}
	  		&F(k,l',l) := \mathbb{P}(\tilde{D}^x_j=k|\tilde{Y}^x_{j-1}=l'+l,\hat{Y}^0_{j-1}=l',\mathcal{F}_j)-\mathbb{P}(\tilde{D}^x_j=k|\tilde{Y}^x_{j-1}=l'+l,\hat{Y}^0_{j-1}=l')\\
	  		&u_{j,l,l',b}=u_j(l'+l,l'+l+b)\text{ and } u_{j,l',b}=u_j(l',l'+b)
	  	\end{align*}
	  where $b$ is in $V$.	With these new notations and \eqref{eqalgo}, we rearrange the sum in \eqref{eq_109} as follows 
	  	\begin{align}\begin{split}\label{eq108*}
	  	\sum_{l,l'}&\sum_k\mathbb{P}(\tilde{D}^x_i=0|\tilde{D}^x_j=k)F(k,l',l)\mathbb{P}(\tilde{Y}^x_{j-1}=l'+l,\hat{Y}^0_{j-1}=l'|\mathcal{F}_{j-1})\\
	  	&=\sum_{l,l'}\mathbb{P}(\tilde{Y}^x_{j-1}=l'+l,\hat{Y}^0_{j-1}=l'|\mathcal{F}_{j-1})\left\lbrace\sum_{k\neq l}F(k,l',l)\mathbb{P}(\tilde{D}^x_i=0|\tilde{D}^x_j=k)+F(l,l',l)\mathbb{P}(\tilde{D}^x_i=0|\tilde{D}^x_j=l)\right\rbrace.
	  	\end{split}
	  	\end{align}
Notice that for $k\neq l$, we have that
\begin{align}\begin{split}\label{eq110*}
		F(k,l',l)&=\sum_{\tiny{b_1,b_2\in V;b_1\neq b_2;b_1-b_2=k-l}}u_{j}(l'+l,l'+l+b_1)u_{j}(l',l'+b_2)-\mathbb{E}\left[u_{j}(l'+l,l'+l+b_1)u_{j}(l',l'+b_2)\right]\\
		&=\sum_{\tiny{b_1,b_2\in V;b_1\neq b_2;b_1-b_2=k-l}}u_{j,l',l,b_1}u_{j,l',b_2}-\mathbb{E}\left[u_{j,l',l,b_1}u_{j,l',b_2}\right].
		\end{split}
	\end{align}
Also, observe that	
\begin{align}\begin{split}\label{eq_108}
		F(l,l',l)&=\sum_{b_1\in V}\left\{u_{j}(l'+l,l'+l+b_1)u_{j}(l',l'+b_1)-\mathbb{E}\left[u_{j}(l'+l,l'+l+b_1)u_{j}(l',l'+b_1)\right]\right\}\\
		&=\sum_{b_1\in V}\left\{u_{j,l',l,b_1}u_{j,l',b_1}-\mathbb{E}\left[u_{j,l',l,b_1}u_{j,l',b_1}\right]\right\},
		\end{split}
	\end{align}
	substituting $u_{j}(l',l'+b_1)$ by $1-\sum_{b_2\in V;b_2\neq b_1}u_{j}(l',l'+b_2)$ into \eqref{eq_108}  we obtain
	\begin{align}\begin{split}\label{eq109*}
	F(l,l',l)&=-\left(\sum_{b_1,b_2\in V;b_1\neq b_2}u_{j}(l'+l,l'+l+b_1)u_{j}(l',l'+b_2)-\mathbb{E}\left[u_{j}(l'+l,l'+l+b_1)u_{j}(l',l'+b_2)\right]\right)\\
	&=-\left(\sum_{b_1,b_2\in V;b_1\neq b_2}u_{j,l',l,b_1}u_{j,l',b_2}-\mathbb{E}\left[u_{j,l',l,b_1}u_{j,l',b_2}\right]\right).
		\end{split}
	\end{align}
Substituting \eqref{eq109*} and \eqref{eq110*} into \eqref{eq108*} we obtain that the sum in \eqref{eq_109}  is equal to
	\begin{align*}	
	&\sum_{l,l'}\mathbb{P}(\tilde{Y}^x_{j-1}=l'+l,\hat{Y}^0_{j-1}=l'|\mathcal{F}_{j-1})\times\\&\bigg(\sum_{b_1,b_2\in V;b_1\neq b_2}\{\mathbb{P}(\tilde{D}^x_i=0|\tilde{D}^x_j=l+b_1-b_2)-\mathbb{P}(\tilde{D}^x_i=0|\tilde{D}^x_j=l)\}\\
	&\hspace{5cm}\times\{u_{j,l',l,b_1}u_{j,l',b_2}-\mathbb{E}\left[u_{j,l',l,b_1}u_{j,l',b_2}\right]\}\bigg)\\
	&=\sum_{b_1,b_2\in V;b_1\neq b_2}\sum_{l',l}\{\mathbb{P}(\tilde{D}^x_i=0|\tilde{D}^x_j=l+b_1-b_2)-\mathbb{P}(\tilde{D}^x_i=0|\tilde{D}^x_j=l)\}\\
	&\hspace{3cm}\times\{u_{j,l',l,b_1}u_{j,l',b_2}-\mathbb{E}\left[u_{j,l',l,b_1}u_{j,l',b_2}\right]\}\times\mathbb{P}(\tilde{Y}^x_{j-1}=l'+l,\hat{Y}^0_{j-1}=l'|\mathcal{F}_{j-1}).
	\end{align*}
	Hence, there is some positive constant $C_2 = C_2(K)$, such that
	\begin{equation*}
		\sum_{j=1}^{A|x|^2}\mathbb{E}\left( \sum^{A|x|^2-1}_{i=j}\left\{\mathbb{P}(\tilde{D}^x_i=0|\mathcal{F}_j)-\mathbb{P}(\tilde{D}^x_i=0|\mathcal{F}_{j-1})\right\}\right)^2 \le C_2\sum_{b_1,b_2\in V;b_1\neq b_2}G(b_1,b_2),
	\end{equation*}
	where
	\begin{align*}
		G(b_1,b_2):=\sum_{j=1}^{A|x|^2}\mathbb{E}\Big(\sum_{i = j}^{A|x|^2 - 1}\sum_{l',l}&\{\mathbb{P}(\tilde{D}^x_i=0|\tilde{D}^x_j=l+b_1-b_2)-\mathbb{P}(\tilde{D}^x_i=0|\tilde{D}^x_j=l)\} \nonumber\\
		&\,\times\{u_{j}(l'+l,l'+l+b_1)u_{j}(l',l'+b_2)-\mathbb{E}\left[u_{j}(l'+l,l'+l+b_1)u_{j}(l',l'+b_2)\right]\} \nonumber\\
		&\times\mathbb{P}(\tilde{Y}^x_{j-1}=l'+l,\hat{Y}^0_{j-1}=l'|\mathcal{F}_{j-1})\Big)^2.
	\end{align*}
	We only analyze $G((1,0),(0,0))$, and the same arguments used for this term work for the other terms  $G(b_1,b_2)$ with $b_1,b_2\in V, b_1 \neq b_2$. Also, to make the notation more compact, we define
	\begin{align*}
	&A^x_{j,l}=\sum^{A|x|^2-1}_{i=j}\left[\mathbb{P}(\tilde{D}^x_i=0|\tilde{D}^x_j=l+(1,0))-\mathbb{P}(\tilde{D}^x_i=0|\tilde{D}^x_j=l)\right]\\
	&\overline{u_{j,l,l'}}=u_{j}(l'+l,l'+l+(1,0))u_{j}(l',l'+(0,0)) -\mathbb{E}\left(u_{j}(l'+l,l'+l+(1,0))u_{j}(l',l'+(0,0))\right)
\\
	&p^x_j(l)= \mathbb{P}(\tilde{Y}^x_{j-1}=l|\mathcal{F}_{j-1}),\,p^0_j(l)= \mathbb{P}(\hat{Y}^0_{j-1}=l|\mathcal{F}_{j-1}).	\end{align*}
    Then
	\begin{equation}\label{eq13**}
	 G((1,0),(0,0)) = \sum^{A|x|^2-1}_{j=1}\mathbb{E}\left(\sum_{l',l}A^x_{j,l}\overline{u_{j,l,l'}}p^x_j(l'+l)p^0_j(l')\right)^2.
	\end{equation}
	In Lemma \ref{LosA}, we proved that the terms $|A^x_{j,l}|$ are bounded uniformly in $l$, $j$, and $x$, and therefore \eqref{eq13**} is bounded by
	\begin{align}\label{tocamirar}
	C^2\sum^{A|x|^2-1}_{j=1}\mathbb{E}\sum_{l',l}\sum_{k,k'}\mathbb{E}(\overline{u_{j,l,l'}}\,\overline{u_{j,k,k'}})p^x_j(l'+l)p^0_j(l')p^x_j(k'+k)p^0_j(k').
	\end{align}
	In \eqref{tocamirar}, we have used the independence of the $u_j 's$ with $\mathcal{F}_{j-1}$.
	Now, if $ \{l+l',l'\}\cap\{k+k',k'\}=\emptyset$, then $\overline{u_{j,l,l'}}\text{ and }\overline{u_{j,k,k'}}$ are independent and the expectation of both terms is zero. When $l+l'=k+k'$ and $k'=l'$, \eqref{tocamirar} becomes
	\begin{align}\label{eq14**}
	\sum^{A|x|^2-1}_{j=1}\mathbb{E}\sum_{l',l}\mathbb{E}(\overline{u_{j,l,l'}}^2)(p^x_j(l'+l))^2(p^0_j(l'))^2 &\leq \sum^{A|x|^2-1}_{j=1}\mathbb{E}\sum_{l',l}p^x_j(l'+l)(p^0_j(l'))^2 \nonumber\\
	&=\sum^{A|x|^2-1}_{j=1}\mathbb{E}\sum_{l'}(p^0_j(l'))^2\nonumber\\
	&=\sum^{A|x|^2-1}_{j=1}\mathbb{P}(\tilde{D}^0_j=0).
	\end{align}
	
		When $l'=k+k'$ and $l+l'=k'$, we have that
	\begin{align}\label{eq15**}
	\sum^{A|x|^2-1}_{j=1}\mathbb{E}\sum_{l',l}\mathbb{E}(\overline{u_{j,l,l'}}\,\overline{u_{j,-l,l+l'}})p^x_j(l'+l)p^0_j(l')p^x_j(l')p^0_j(l+l')&\leq \sum^{A|x|^2-1}_{j=1}\mathbb{E}\sum_{l',l}p^0_j(l')p^x_j(l')p^0_j(l+l')\nonumber\\
	&=\sum^{A|x|^2-1}_{j=1}\mathbb{E}\sum_{l'}p^0_j(l')p^x_j(l')\nonumber\\
	&=\sum^{A|x|^2-1}_{j=1}\mathbb{P}(\tilde{D}^x_j=0).
	\end{align}
	For the case $l'+l=k'+k$ and $k'\neq l'$ we obtain
	\begin{align}\label{eq16**}
	\sum^{A|x|^2-1}_{j=1}\mathbb{E}\sum_{l',l}\sum_{k'}\mathbb{E}(\overline{u_{j,l,l'}}\,\overline{u_{j,l'+l-k',k'}})p^x_j(l'+l)^2p^0_j(l')p^0_j(k')\nonumber &\leq \sum^{A|x|^2-1}_{j=1}\mathbb{E}\sum_{l',l}p^x_j(l')^2p^0_j(l'+l)\nonumber\\&=\sum^{A|x|^2-1}_{j=1}\mathbb{E}\sum_{l'}p^x_j(l')^2\nonumber\\
	&=\sum^{A|x|^2-1}_{j=1}\mathbb{P}(\tilde{D}^0_j=0).
	\end{align}
	Finally, when $l'+l=k'$ and $k'+k\neq l'$ 
	\begin{align}\label{eq17**}
	\sum^{A|x|^2-1}_{j=1}\mathbb{E}\sum_{l',l}\sum_{k}\mathbb{E}(\overline{u_{j,l,l'}}\,\overline{u_{j,k,l+l'}})p^x_j(l+l')p^0_j(l+l')p^0_j(l')p^x_j(l+l'+k)&\leq \sum^{A|x|^2-1}_{j=1}\mathbb{E}\sum_{l,l'}p^x_j(l+l')p^0_j(l+l')p^0_j(l')\nonumber\\
	&= \sum^{A|x|^2-1}_{j=1}\mathbb{E}\sum_{l,l'}p^x_j(l)p^0_j(l)p^0_j(l-l')\nonumber\\
	&= \sum^{A|x|^2-1}_{j=1}\mathbb{E}\sum_{l}p^x_j(l)p^0_j(l)\nonumber\\
	&=\sum^{A|x|^2-1}_{j=1}\mathbb{P}(\tilde{D}_j^x=0).
	\end{align}
	Reasoning as in the proof of Proposition 2.3 in \cite{FerrariFontes} we have that
	\begin{equation*}
		\mathbb{P}\left(D_n = 0|D_0 = x\right) = \mathbb{P}(\tilde{D}_n = x), \, \text{ for } n\ge 0 \text{ and } x\in\mathbb{Z}^d.
	\end{equation*} Hence, equation \eqref{eq11**}  implies that  \eqref{eq14**} and \eqref{eq16**} are both $O(\ln(A|x|^2))$. Using \eqref{eq8**}, we conclude that \eqref{eq15**} and \eqref{eq17**} are also $O(\ln(A|x|^2))$. Summing up, the variance of the sum at the left hand side of \eqref{eq19} is bounded by an $O(\ln(A|x|^2))$. Therefore, the variance of the whole term in \eqref{eq19} is bounded by an $o(1)$, and we have proved the result.
\end{proof}

\bigskip

\noindent{\bf Acknowledgements.} We would like to thank Hubert Lacoin for pointing us in a good direction on the issue of the Gaussianity of 
the invariant distribution of the RAP considered in this paper, as discussed at the introduction.

\bibliographystyle{acm}
\addcontentsline{toc}{part}{Bibliography}


\end{document}